\documentclass{article}%
\usepackage{amssymb}
\usepackage{amsfonts}
\usepackage{amsmath}
\usepackage{graphicx}
\usepackage{xcolor}%
\providecommand{\U}[1]{\protect\rule{.1in}{.1in}}
\newtheorem{theorem}{Theorem}

\newtheorem{corollary}[theorem]{Corollary}

\newtheorem{definition}[theorem]{Definition}

\newtheorem{lemma}[theorem]{Lemma}

\newtheorem{proposition}[theorem]{Proposition}
\newtheorem{remark}[theorem]{Remark}

\newenvironment{proof}[1][Proof]{\noindent \textbf{#1.} }{\  \rule{0.5em}{0.5em}}
\begin{document}
	
	\title{\textbf{The relaxed maximum principle for G-stochastic control systems with
			controlled jumps }}
	\author{H. B. Gherbal\thanks{Laboratory of Mathematical Analysis, Probability and
			Optimization, University of Mohamed Khider, P.O.Box 145, Biskra 07000,
			Algeria. Email: h.bengherbal@yahoo.com}
		\and A. Redjil \thanks{Lab. of Probability and Statistics (LaPS) Department of
			Mathematics, Badji Mokhtar University B.P.12 23000, Annaba, Algeria. E-mail:
			$rdj\_amel@yahoo.fr$}
		\and O. Kebiri \thanks{Corresponding author: Brandenburgische Technische
			Universit\"{a}t Cottbus-Senftenberg, Germany, E-mail: omar.kebiri@b-tu.de,
			Tel: +49 (0) 355 69 2327. Fax: +49 (0) 355 69 3595} }
	\date{}
	\maketitle
	
	\begin{abstract}
		This paper is concerned with optimal control of systems driven by G-stochastic
		differential equations (G-SDEs), with controlled jump term. We study the
		relaxed problem, in which admissible controls are measure-valued processes and
		the state variable is governed by an G-SDE\ driven by a counting measure
		valued process called relaxed Poisson measure such that the compensator is a
		product measure. Under some conditions on the coefficients, using the
		G-chattering lemma, we show that the strict and the relaxed control problems
		have the same value function. Additionally, we derive a maximum principle for
		this relaxed problem.
		
	\end{abstract}
	
	
	
	\section{\textbf{Introduction}}
	
	We consider a stochastic control problem where the state variable is a
	solution of a SDE driven by a $G$-Brownian motion with jumps, the control
	enters both the drift and the jump term. More precisely the system evolves
	according to the SDE
	\begin{equation}
		\left\{
		\begin{array}
			[c]{l}%
			{\small dx}_{t}{\small =b(t,x}_{t}{\small ,u}_{t}{\small )dt+\sigma(t,x}%
			_{t}{\small )dB}_{{\small t}}{\small +\gamma(t,x_{t},u_{t})d\langle B\rangle
			}_{t}{\small +}%
			{\displaystyle\int\limits_{\Gamma}}
			{\small f(t,x}_{t^{-}}{\small ,\theta,u}_{t}{\small )}\widetilde{{\small N}%
			}{\small (dt,d\theta)}\\
			{\small x}_{0}{\small =x}%
		\end{array}
		\right.  , \label{SDE1}%
	\end{equation}

	on some space of sublinear expectation $(\Omega,H,\widehat{\mathbb{E}%
	},\mathbb{F}^{\mathcal{P}})$, where $\mathbb{F}^{\mathcal{P}}$ is the
	universal filtration$,$ and $b,$ $\sigma,$ $\gamma,$ $f$ are given
	deterministic functions, $u$ is the control process. We consider here an
	independent Poisson random measure $N$; whose compensator is given by
	$v(d\theta)dt.$
	
	The expected cost to be minimized over the class of admissible controls is
	defined by:
	\begin{equation}
		\label{functional}J(x;u)=\sup_{\mathbb{P}\in\mathcal{P}}\mathbb{E}%
		^{\mathbb{P}}\left[  g(x_{T})+
		{\displaystyle\int\limits_{0}^{T}}
		h(t,x_{t},u_{t})dt\right]  =\widehat{E}\left[  g(x_{T})+
		{\displaystyle\int\limits_{0}^{T}}
		h(t,x_{t},u_{t})dt\right]  ,
	\end{equation}
	where $x$ is the initial condition of the process $(x_{t})_{t\in[0,T]}$
	
	We defined then the value function $V$ by:
	\begin{equation}
		\label{valueFunction}V(x):=\inf_{u\in\mathcal{U}}J(x;u),
	\end{equation}
	where $\mathcal{U}$ is the set of admissible controls
	
	A control process that verify \eqref{valueFunction} is called optimal.
	
	In the recent years the framework of G-expectation has found increasing
	application in the domain of finance and economics, e.g., Epstein and Ji
	\cite{EJ13,EJ14} study the asset pricing with ambiguity preferences, Beissner
	\cite{B13} who studies the equilibrium theory with ambiguous volatility, and
	many others see e.g. \cite{YH16,BD18,YH18}, also see
	\cite{KNH18a,KNH18b,HKNR19}.for numerical methods. The motivation is that many
	systems are subject to model uncertainty or ambiguity due to incomplete
	information, or vague concepts and principles. Aspects of model ambiguity such
	as volatility uncertainty have been studied by Peng (2007, 2008, 2010,
	\cite{P07,P08,P10}) who introduced a sublinear expectation with a process
	called $G$-Brownian motion, also by Denis and Martini \cite{DM06} who
	suggested a structure based on quasi-sure analysis from abstract potential
	theory to construct a similar structure using a tight family $\mathcal{P}$ of
	possibly mutually singular probability measures.
	
	The strict control problem may fail to have an optimal solution, if we don't
	impose some kind of convexity assumption. In this case, we must embed the
	space of strict controls into a larger space that has nice properties of
	compactness and convexity. This space is that of probability measures on $A$,
	where $A$ is the set of values taken by the strict control. These measures
	valued processes are called relaxed controls. In the classical framework, the
	first existence result of an optimal relaxed control is proved by Fleming
	\cite{Fl}, for the SDEs with uncontrolled diffusion coefficient and no jump
	term. For such systems of SDEs a maximum principle has been established in
	\cite{BDM, BahDjeMez, BahMez}. The case where the control variable appears in
	the diffusion coefficient has been solved in \cite{El-karoui}. The existence
	of an optimal relaxed control of SDEs, where the control variable enters in
	the jump term was derived by Kushner \cite{Kushner-article}, also recently the
	work given by H. B Gherbal and B.\ Mezerdi in 2017 \cite{HBG} of relaxed
	stochastic maximum principle in optimal control of diffusions with controlled
	jumps, for which the state variable is governed by a SDE driven by a counting
	measure valued process called relaxed Poisson measure, where the existence of
	an optimal relaxed control and a Pontryagin maximum principle were proved.
	
	\bigskip
	
	In the $G$-framework the existence of an optimal relaxed control is
	established In 2018, by Redjil and Choutri \cite{RC18}, where a stochastic
	differential equation is considered without jump term and an uncontrolled
	diffusion coefficient, the celebrate Chattering lemma was generalized in the
	$G$-framework and the existence of relaxed optimal control was proved. The
	same result in the case with jump term is recently proved by A. Redjil, H. B.
	Gherbal and O. Kebiri \cite{RBK}.
	
	In this paper, we establish a Pontryagin maximum principle for the relaxed
	control problem given by (\ref{SDE}) and (\ref{cost}). More precisely we
	derive necessary conditions for optimality satisfied by an optimal control.
	The proof is based on the results obtained in \cite{RBK}, Pontryagin's maximum
	principle for nearly optimal strict controls and some stability results of
	trajectories and adjoint processes with respect to the control variable.
	
	The motivation of our work came from e.g. applications in finance when a jump
	process models the stock price where we can't estimate exactly the
	coefficients of the noise. The uncertainty in the noise coefficient will
	produce a G-SDE with jump, if then we want to control this dynamic, this lead
	a controlled G-SDE with jump.
	
	The rest of the paper is organized as follows: in section 2, we formulate the
	control problem, and introduce the assumptions of the model. Section 3 is
	devoted to the proof of the approximation and stability results. In the last
	section, we state and prove a maximum principle for our relaxed control
	problem, which is the main result of this paper.
	
	\section{Formulation of the problem}
	
	\subsection{G-Strict control problem}
	
	We consider a control problem of systems governed by stochastic differential
	equations on some sublinear expectation space $(\Omega,H$,$\widehat{E}%
	,F^{\mathcal{P}})$, such that $F^{\mathcal{P}}$ the universal filtration
	defined by $F^{\mathcal{P}}=\left\{  \widehat{\mathcal{F}}_{t}^{\mathcal{P}%
	}\right\}  _{t\geq0},$ where $\widehat{\mathcal{F}}_{t}^{\mathcal{P}}=%
	{\displaystyle\bigcap\limits_{P\in\mathcal{P}}}
	\left(  \mathcal{F}_{t}^{P}\vee\mathcal{N}_{\mathcal{P}}\right)  $ for
	$t\geq0,$ such that $\mathcal{F}_{t}^{P}$.generated by a $G-$Brownian motion
	$B$ and an independent Poisson measure $N$, with compensator $\nu(d\theta)dt$,
	where the jumps are confined to a compact set $\Gamma.$ and set
	\[
	\widetilde{N}{\small (dt,d\theta)=N(dt,d\theta)-\nu(d\theta)dt}%
	\]

	Consider a compact set $A$ in $%
	\mathbb{R}
	^{k}$ and let $\mathcal{U}$ the class of measurable, adapted processes
	$u:\left[  0;T\right]  \times\Omega\longrightarrow A$, such that $u\in
	M_{G}^{2}\left(  0,T\right)  .$ For any $u$ $\in\mathcal{U}$, we consider the
	following stochastic differential equation (SDE)
	\begin{equation}
		\left\{
		\begin{array}
			[c]{l}%
			{\small dx}_{t}{\small =b(t,x}_{t}{\small ,u}_{t}{\small )dt+\sigma(t,x}%
			_{t}{\small )dB}_{{\small t}}{\small +\gamma(t,x_{t},u_{t})d\langle B\rangle
			}_{t}{\small +}%
			{\displaystyle\int\limits_{\Gamma}}
			{\small f(t,x}_{t^{-}}{\small ,\theta,u}_{t}{\small )}\widetilde{{\small N}%
			}{\small (dt,d\theta)}\\
			{\small x}_{0}{\small =x,}%
		\end{array}
		\right.  \label{SDE}%
	\end{equation}

	where
	\begin{align*}
		b  &  :\left[  0;T\right]  \times%
		\mathbb{R}
		^{n}\times A\longrightarrow%
		\mathbb{R}
		^{n}\\
		\sigma &  :\left[  0;T\right]  \times%
		\mathbb{R}
		^{n}\longrightarrow\mathcal{M}_{n\times d}(%
		\mathbb{R}
		)\\
		\gamma &  :\left[  0;T\right]  \times%
		\mathbb{R}
		^{n}\times A\longrightarrow%
		\mathbb{R}
		^{n\times n}\\
		f  &  :\left[  0;T\right]  \times%
		\mathbb{R}
		^{n}\times\Gamma\times A\longrightarrow%
		\mathbb{R}
		^{n}%
	\end{align*}

	are bounded, measurable and continuous functions.
	
	The expected cost is given by%
	\begin{equation}
		J(u)=\sup_{\mathbb{P}\in\mathcal{P}}\mathbb{E}^{\mathbb{P}}\left[  g(x_{T})+%
		{\displaystyle\int\limits_{0}^{T}}
		h(t,x_{t},u_{t})dt\right]  =\widehat{E}\left[  g(x_{T})+%
		{\displaystyle\int\limits_{0}^{T}}
		h(t,x_{t},u_{t})dt\right]  \label{cost}%
	\end{equation}

	where
	\begin{align*}
		g  &  :%
		\mathbb{R}
		^{n}\longrightarrow%
		\mathbb{R}%
		\\
		h  &  :\left[  0;T\right]  \times%
		\mathbb{R}
		^{n}\times A\longrightarrow%
		\mathbb{R}%
	\end{align*}

	be bounded and continuous functions.
	
	The problem is to minimize the functional $J(.)$ over $\mathcal{U}$. A control
	that solves this problem is called optimal.
	
	We note that from the result of \cite{Pa12}, equation (\ref{SDE}) has a unique
	solution, under these assumptions (A) :
	
	(A1) Let be $b,$ $\sigma,$ $\gamma$ and $f$ bounded and Lipschitz continuous
	with respect to the state variable $x$ uniformly in $(t,u)$, also we suppose
	that $\gamma(t,x,.)$ is a symmetric $d\times d$ matrix with each element.
	
	(A2) For all $(t,x,\theta)\in\lbrack0,T]\times\mathbb{R}^{n}\times\Gamma$ the
	functions $b(t,x,.),$ $f(t,x,\theta,.)$ and $\gamma(t,x,.)$ are continuous in
	$u\in U$.
	
	(A3) $b(.,x,.)$ and $\gamma(.,x,.)$ $\ $and $\sigma(.,x)$ taking value in
	$M_{G}^{2}(0,T)$ and $f(.,x,.,.)$ takes value in $\hat{H}_{G}^{2}(0,T)$
	
	(A4) The functions $g$ and $h(.,x,.)$ are taking value in $M_{G}^{2}(0,T)$ and
	bounded. Moreover we suppose that $g$ is Lipschitz continuous, and $h$ is
	Lipschitz continuous with respect to the state variable $x$ uniformly in time
	and control $(t,u).$
	
	\subsection{The G-relaxed control problem}
	
	Let $(A,d)$ be a separable metric space and $\mathcal{P}(A)$ be the space of
	probability measures on the set $A$ endowed with its Borel $\sigma$-algebra
	$\mathcal{B}(A)$. The class $M([0,T]\times A)$ of relaxed controls we consider
	in this paper is a subset of the set $M([0,T]\times A)$ of Radon measures
	$\nu(dt,da)$ on $[0,T]\times A$ equipped with the topology of stable
	convergence of measures, whose projections on $[0,T]$ coincide with the
	Lebesgue measure $dt$, and whose projection on $A$ coincide with some
	probability measure $\mu_{t}(da)\in\mathcal{P}(A)$ i.e. $\nu(da,dt):=\mu
	_{t}(da)dt.$ The topology of stable convergence of measures is the coarsest
	topology which makes the mapping
	\[
	q\mapsto\int_{0}^{T}\int_{A}\varphi(t,a)q(dt,da)
	\]
	continuous, for all bounded measurable functions $\varphi(t,a)$ such that for
	fixed $t$, $\varphi(t,\cdot)$ is continuous. Equipped with this topology,
	$M:=M([0,T]\times A)$ is a separable metrizable space. Moreover, it is compact
	whenever $A$ is compact. The topology of stable convergence of measures
	implies the topology of weak convergence of measures. For further details see
	\cite{El-karoui,ENJ88}.
	
	Now we present the following definitions:
	
	\begin{definition}
		$\text{Lip}(\Omega)$ is the set of random variables of the form $\xi
		:=\varphi(B_{t_{1}},B_{t_{2}},\ldots,B_{t_{n}})$ for some bounded Lipschitz
		continuous function $\phi$ on $%
		\mathbb{R}
		^{d\times n}$ and $0\leq t_{1}\leq t_{2}\leq\cdots\leq t_{n}\leq T$. The
		coordinate process $(B_{t},\,\,t\geq0)$ is called $G$-Brownian motion whenever
		$B_{1}$ is $G$-normally distributed under $\widehat{\mathbb{E}}$ $[\cdot]$ and
		for each $s,t\geq0$ and $t_{1},t_{2},\ldots,t_{n}\in\lbrack0,t]$ we have
		\[
		\widehat{\mathbb{E}}[\varphi(B_{t_{1}},\ldots,B_{t_{n}},B_{t+s}-B_{t}%
		)]=\widehat{\mathbb{E}}[\psi(B_{t_{1}},\ldots,B_{t_{n}})],
		\]
		where $\psi(x_{1},\ldots,x_{n})=\widehat{\mathbb{E}}[\varphi(x_{1}%
		,\ldots,x_{n},\sqrt{s}B_{1})]$. This property implies that the increments of
		the $G$-Brownian motion are independent and that $B_{t+s}-B_{t}$ and $B_{s}$
		are identically $N(0,s\Sigma)$-distributed.
	\end{definition}
	
	Next, we introduce the class of relaxed stochastic controls on $(\Omega
	_{T},\mathcal{H},\widehat{\mathbb{E}})$, where $\mathcal{H}$ is a vector
	lattice of real functions on $\Omega$ such that $\text{Lip}(\Omega_{T}%
	)\subset\mathcal{H}$.
	
	\begin{definition}
		A relaxed stochastic control on $(\Omega_{T},Lip(\Omega_{T}%
		),\widehat{\mathbb{E}})$ is a random measure $q(\omega,dt,da)=\mu_{t}%
		(\omega,da)dt$ such that for each subset $C\in\mathcal{B}(A)$, the process
		$(\mu_{t}(C))_{t\in\lbrack0,T]}$ is $\mathbb{F}^{\mathcal{P}}$-progressively
		measurable i.e. for every $t\in\lbrack0,T]$, the mapping $[0,t]\times
		\Omega\rightarrow\lbrack0,1]$ defined by $(s,\omega)\mapsto\mu_{s}(\omega,A)$
		is $B([0,t])\otimes\widehat{\mathcal{F}}_{t}^{\mathcal{P}}$-measurable. In
		particular, the process $(\mu_{t}(C))_{t\in\lbrack0,T]}$ is adapted to the
		\textit{universal } filtration $\mathbb{F}^{\mathcal{P}}$. We denote by
		$\mathcal{R}$ the class of relaxed stochastic controls.
	\end{definition}
	
	The set $\mathcal{U}\left[  0,T\right]  $ of strict controls constituted of
	$\mathbb{F}^{\mathcal{P}}$-adapted processes $u$ taking values in the set $A$,
	embeds into the set $\ \mathcal{R}$ of relaxed controls through the mapping:
	\[
	\Phi:\quad\mathcal{U}\left[  0,T\right]  \ni u\mapsto\Phi(u)(dt,da)=\delta
	_{u(t)}(da)dt\in\mathcal{R}.
	\]

	\begin{definition}
		Let $\mu$ a relaxed \ representation of an admissible control $u$ , for each
		$\Gamma_{0}\subset\Gamma,$ $\Gamma_{0}$ is a Borel set $(\Gamma_{0}\in
		B(\Gamma))$ and $A_{0}\subset A$, $(A_{0}\in\mathcal{B}(A)),$ we define:%
		
		\[
		N^{^{\mu}}(\left[  0,t\right]  ,A_{0},\Gamma_{0}%
		)\underset{\mbox{{\tiny we denote by}}}{:=}N^{^{\mu}}(t,A_{0},\Gamma_{0}%
		)=\int_{0}^{t}\int_{\Gamma_{0}}1_{A_{0}}(u(s)).N(ds,d\theta).
		\]

		$N^{^{\mu}}$ is the number of jumps of $\int_{0}^{t}\int_{\Gamma_{0}}\theta
		N(ds,d\theta)$ on $\left[  0,t\right]  $ with values in $\Gamma_{0}$ and
		where: $u(s)\in A_{0}$ at the jump times $s.$
		
		Since \ \ $1_{A_{0}}(u(s))=\mu_{s}(A_{0})$; then \ the compensator of the
		counting measure valued process $N^{^{\mu}}$is: \ \ \ \ \ \
		\[
		v(d\theta)\mu_{t}(da)dt=\mu_{t}\otimes v(da,d\theta)
		\]
		
	\end{definition}
	
	\begin{definition}
		A relaxed Poisson measure $N^{^{\mu}}$ is a counting measure valued process
		such that its compensator is the product measure of the relaxed control $\mu$
		with the compensator $v$ of $N$ such that : For each $\Gamma_{0}\subset
		\Gamma,$ $\Gamma_{0}$ is a Borel set $(\Gamma_{0}\in\mathcal{B}(\Gamma))$ and
		$A_{0}\subset A$, $(A_{0}\in\mathcal{B}(A)),$ the processes:%
		
		\[
		Z^{\mu}=\widetilde{N}^{^{\mu}}(t,A_{0},\Gamma_{0})=N^{^{\mu}}(t,A_{0}%
		,\Gamma_{0})-\mu(t,A_{0})\nu(\Gamma_{0})
		\]
		are $\widehat{\mathcal{F}}_{t}^{\mathcal{P}}-$martingales and orthogonal for
		disjoint $\Gamma_{0}\times A_{0}$, because according to \cite{HBG}, the
		processes $Z^{\mu}$ \ are a $\mathcal{F}_{t}^{\mathbb{P}}$- martingales for
		each $\mathbb{P}\in\mathcal{P},$ and so is an $\widehat{\mathcal{F}}%
		_{t}^{\mathcal{P}}$-martingale, also are orthogonal for disjoint $\Gamma
		_{0}\times A_{0}$.
	\end{definition}
	
	\begin{proposition}
		For any bounded measurable function $\varphi$ with real values, the process
		$Y$ given by:%
		
		\[
		\int_{0}^{t}\int_{\Gamma}\int_{A}\varphi(s,x_{s^{-}},\theta,a)N^{\mu
		}(dt,d\theta,da)-\int_{0}^{t}\int_{\Gamma}\int_{A}\varphi(s,x_{s^{-}}%
		,\theta,a)v(d\theta)\mu_{s}(da)ds
		\]

		is an $\widehat{\mathcal{F}}_{t}^{\mathcal{P}}-$martingale.
	\end{proposition}
	
	\begin{proof}
		From H. Ben Gherbal and B. Mezerdi, \cite{HBG} the process $Y$ is a
		$F_{t}^{\mathbb{P}}$-martingale for each $\mathbb{P}\in\mathcal{P},$ we deduce
		that $Y$ is an $\widehat{\mathcal{F}}_{t}^{\mathcal{P}}$- martingale.
	\end{proof}
	
	\begin{proposition}
		Consider a sequence of $(\mu_{s}^{n}\otimes\nu)_{n}$ converging weakly to
		$\mu_{s}\otimes\nu$ on $\Omega\times\left[  0,T\right]  \times A\times\Gamma,$
		there exists a sequence of orthogonal martingale measures $\widetilde{N}^{n}$
		defined on $\Omega\times\left[  0,T\right]  \times A\times\Gamma,$ such that
		for each bounded function $\varphi:$
		
		\bigskip%
		\[
		\int_{0}^{t}\int_{A}\int_{\Gamma}\varphi(s,X_{s^{-}}^{\mu},\theta
		,a)\widetilde{N}^{n}(ds,d\theta,da)\underset{n\rightarrow\infty
		}{\longrightarrow}\int_{0}^{t}\int_{A}\int_{\Gamma}\varphi(s,X_{s^{-}}^{\mu
		},\theta,a)\widetilde{N}^{\mu}(ds,d\theta,da)\text{\ quasi-surely}%
		\]
		
	\end{proposition}
	
	\begin{proof}
		Given a fixed probability measure
		\[
		\mathbb{P}\in\mathcal{P},\underset{{\tiny \cite{HBG}}}{\implies}\int_{0}%
		^{t}\int_{A}\int_{\Gamma}\varphi(s,x_{s^{-}}^{\mu},\theta,a)\widetilde{N}%
		^{n}(ds,d\theta,da)\rightarrow\int_{0}^{t}\int_{A}\int_{\Gamma}\varphi
		(s,x_{s^{-}}^{\mu},\theta,a)\widetilde{N}^{\mu}(ds,d\theta,da),\text{
		}\mathbb{P}\mbox{-surely}
		\]
		this means that we have the convergence outside a polar set, which means that
		we have quasi surely convergence.
	\end{proof}
	
	\subsection{\bigskip G-relaxed control version of the G-SDE with controlled
		jumps}
	
	Now we present our relaxed controlled system:
	
	The $G$-SDE\ with controlled jumps in terms of relaxed Poisson measure is
	given by:%
	
	\begin{equation}
		\left\{
		\begin{array}
			[c]{l}%
			dx^{\mu}(t)=\int_{A}b(t,x_{t}^{\mu},a)\mu_{t}(da)dt+\sigma(t,x_{t}^{\mu
			})dB_{t}+\int_{A}\gamma(t,x_{t}^{\mu},a)\mu_{t}(da)d\langle B\rangle_{t}+\\
			\int_{A}\int_{\Gamma}f(t,x_{t^{-}}^{\mu},\theta,a)\widetilde{N}^{\mu
			}(dt,d\theta,da)\\
			x_{0}^{\mu}=0
		\end{array}
		\right.  \label{relaxedSDE}%
	\end{equation}

	The cost functional is given by:
	\[
	J(\mu)=\widehat{\mathbb{E}}\left[  \int_{0}^{T}\int_{A}h(t,x_{t}^{\mu}%
	,a)\mu_{t}(da)\,dt+g(x_{T}^{\mu})\right]  .
	\]

	\section{Approximation of trajectories and stability results}
	
	The next lemma, which called G-chattering lemma gives the approximation of a
	relaxed control by a sequence of strict controls order for the relaxed control
	problem. This result is considered essential in showing that the relaxed
	control problem is a truly an extension of the strict one. we refer to
	\cite{RBK} to more detail of this subsection.
	
	\begin{lemma}
		\label{chattlemma}Let $(U,d)$ be a separable metric space and assume that $U$
		is a compact set. Let $(\mu_{t})_{t}$ be an $\mathbb{F}^{\mathbb{P}}$
		-progressively measurable process \ with values in $P(U)$. Then there exists a
		sequence $(u_{t}^{n})_{n\geq0}$ of $\mathbb{F}^{\mathbb{P}}$-progressively
		measurable processes with values in $U$ such that the sequence of random
		measures $\delta_{u_{t}^{n}}(da)dt$ converges in the sense of stable
		convergence (thus, weakly) to $\mu_{t}(da)dt$ quasi-surely : \
		\[
		\mu_{t}^{n}(da)dt=\delta_{u_{t}^{n}}(da)dt\longrightarrow\mu_{t}(da)dt\text{
			\ \ \ quasi-surely}%
		\]
		
	\end{lemma}
	
	\begin{proof}
		see \cite{RC18}
	\end{proof}
	
	\begin{lemma}
		Under our assumption (A), for every $\ \mathbb{P}\in\mathcal{P}$, it holds that:
		
		\begin{enumerate}
			\item
			\begin{equation}
				\lim_{n\rightarrow\infty}\mathbb{E}^{\mathbb{P}}\left[  \sup_{0\leq t\leq
					T}\left\vert x^{n}(t)-x^{\mu}(t)\right\vert ^{2}\right]  =0, \label{diff}%
			\end{equation}

			\textit{and,}
			\begin{equation}
				\lim_{n\rightarrow\infty}J^{\mathbb{P}}(u^{n})=J^{\mathbb{P}}(\mu).
				\label{equa1}%
			\end{equation}

			\item \textit{Moreover},
			\begin{equation}
				\inf_{u\in\mathcal{U}[0,T]}J^{\mathbb{P}}(u)=\inf_{\mu\in\mathcal{R}%
				}J^{\mathbb{P}}(\mu), \label{equa2}%
			\end{equation}

			and there exists a relaxed control $\hat{\mu}_{\mathbb{P}}\in\mathcal{R}$ such
			that \textit{\ } $J^{\mathbb{P}}(\hat{\mu}_{\mathbb{P}})=\inf_{\mu
				\in\mathcal{R}}J^{\mathbb{P}}(\mu).$
		\end{enumerate}
		
		\begin{proof}
			See \cite{RBK}
		\end{proof}
	\end{lemma}
	
	The next theorem gives the stability of the stochastic differential equations
	with respect to the control variable, and that the two problems has the same
	infimum of the expected costs.
	
	\begin{theorem}
		\label{stability}Under our assumption (A) we have:
		
		\begin{enumerate}
			\item Let $\mu$ be a relaxed control \ and let $x^{\mu}$ the corresponding
			trajectory. Then there exists a sequence $(u^{n})$ of strict controls such that:%
			
			\[
			\lim_{n\longrightarrow\infty}\widehat{\mathbb{E}}\left[  \sup_{0\leq t\leq
				T}\left\vert x_{t}^{n}-x_{t}^{\mu}\right\vert ^{2}\right]  =0,
			\]

			where $x_{t}^{n}$ denotes the trajectory associated to $u^{n}.$
			
			\item Let $J(u^{n})$ and $J(\mu)$ be the cost functional corresponding
			respectively to $u^{n}$ and $\mu$ (where $dt\delta_{u^{n}}(t)(da)$ converges
			weakly to $dt\mu_{t}(da)$ quasi-surely). Then, there exists a subsequence
			$\left(  u^{n_{k}}\right)  $ of $\left(  u^{n}\right)  $ such that%
			\[
			\lim_{k\longrightarrow\infty}J((u^{n_{k}})=J(\mu).
			\]
			
		\end{enumerate}
	\end{theorem}
	
	\begin{proof}
		See \cite{RBK}.
	\end{proof}
	
	\section{Maximum principle for relaxed control problems}
	
	Our main goal in this section is to establish optimality necessary conditions
	for relaxed control problems, where the system is described by a G-SDE driven
	by a relaxed Poisson measure. The proof is based on the G-chattering lemma, we
	derive necessary conditions of near optimality satisfied by a sequence of
	strict controls. By using stability properties of the state equations and
	adjoint processes, we obtain the maximum principle for our relaxed problem.
	
	\subsection{The maximum principle for strict control}
	
	Under the above hypothesis, (\ref{SDE}) has a unique strong solution and the
	cost functional (\ref{cost}) is well defined from $U$ into $%
	\mathbb{R}
	,$ for more detail see \cite{RBK}. The purpose of this subsection is to derive
	optimality necessary conditions, satisfied by an optimal strict control. The
	proof is based on the strong perturbation of the optimal control $u^{\ast},$
	which defined by :
	\[
	u^{h}=\left\{
	\begin{array}
		[c]{l}%
		\nu\text{ \ \ \ if \ }t\in\left[  t_{0};t_{0}+h\right] \\
		u^{\ast}\text{ \ \ \ \ \ \ otherwise }%
	\end{array}
	\right.
	\]

	where $0\leq t_{0}<T$ is fixed, $h$ is sufficiently small, and $\nu$ is an
	arbitrary $A-$valued $\mathcal{F}_{t_{0}}-$measurable random such that
	$E\left\vert \nu\right\vert ^{2}<\infty.$ Let $x_{t}^{h}$ denotes the
	trajectory associated with $u^{h},$ then%
	
	\[
	\left\{
	\begin{array}
		[c]{l}%
		x_{t}^{h}=x_{t}^{\ast}\text{ \ \ \ };\text{ }t\leq t_{0}\\
		dx_{t}^{h}=b(t,x_{t}^{h},\nu)dt+\sigma(t,x_{t}^{h})dB_{t}{\small +\gamma
			(t,x_{t}^{h},\nu)d\langle B\rangle}_{t}+%
		{\displaystyle\int\limits_{\Gamma}}
		{\small f(t,x}_{t^{-}}^{h}{\small ,\theta,\nu)}\widetilde{{\small N}%
		}{\small (dt,d\theta)}\text{ \ \ };t_{0}<t<t_{0}+h\\
		dx_{t}^{h}=b(t,x_{t}^{h},u^{\ast})dt+\sigma(t,x_{t}^{h})dB_{t}{\small +\gamma
			(t,x_{t}^{h},u^{\ast})d\langle B\rangle}_{t}+%
		{\displaystyle\int\limits_{\Gamma}}
		{\small f(t,x}_{t^{-}}^{h}{\small ,\theta,u^{\ast})}\widetilde{{\small N}%
		}{\small (dt,d\theta)}\text{ \ \ };t_{0}+h<t<T
	\end{array}
	\right.
	\]

	We first have
	
	\begin{lemma}
		\label{xhxetoile}Under assumptions (H$_{1}$)-(H$_{2}$), we have
	\end{lemma}
	
	For every$\ \mathbb{P}\in\mathcal{P}$, it holds that%
	\begin{equation}
		\lim_{h\rightarrow0}\mathbb{E}^{\mathbb{P}}\left[  \sup_{0\leq t\leq
			T}\left\vert x_{t}^{h}-x_{t}^{\ast}\right\vert ^{2}\right]  =0, \label{everyp}%
	\end{equation}
	and,
	\begin{equation}
		\lim_{h\rightarrow0}\widehat{E}\left[  \sup_{t\in\left[  t_{0};T\right]
		}\left\vert x_{t}^{h}-x_{t}^{\ast}\right\vert ^{2}\right]  =0 \label{estx}%
	\end{equation}

	\begin{proof}
		Under every $\mathbb{P}\in\mathcal{P}$, the G-SDEs (\ref{SDE}) and
		(\ref{relaxedSDE}) becomes standard SDEs driven by a standard Brownian motion
		$B$ and a Poisson measure $\widetilde{N}$, the proof of (\ref{everyp}) follows
		from H.Ben Gherbal and B.Mezerdi \cite{HBG}, we sketch it here. Using the fact
		that under $\mathbb{P}\in\mathcal{P}$, $\widetilde{N}$ is a martingale and $B$
		is a continuous martingale whose quadratic variation process $\langle
		B\rangle$ is such that $\pi_{t}=\frac{d\langle B\rangle_{t}}{dt}$ is bounded
		by a deterministic $d\times d$ symmetric positive definite matrix
		$\overline{\sigma},$and $x^{h}$ satisfy
		\[
		dx_{t}^{h}=b(t,x_{t}^{h},\nu)dt+\sigma(t,x_{t}^{h})dB_{t}{\small +\pi
			_{t}\gamma(t,x_{t}^{h},\nu)dt}+%
		{\displaystyle\int\limits_{\Gamma}}
		{\small f(t,x}_{t^{-}}^{h}{\small ,\theta,\nu)}\widetilde{{\small N}%
		}{\small (dt,d\theta),}%
		\]
		then the result gives by a standard arguments from stochastic calculus, for
		more detail see H.Ben Gherbal and Mezerdi \cite{HBG}.
		
		For the second limit, set
		\[
		\varsigma_{h}=\sup_{t\in\left[  t_{0};T\right]  }\left\vert x_{t}^{h}%
		-x_{t}^{\ast}\right\vert ^{2},
		\]
		if there is a $\theta>0$ such that $\widehat{E}\left[  \varsigma_{h}\right]
		\geq\theta,$ we can find a probability $\mathbb{P}\in\mathcal{P}$ such that
		$\widehat{E}\left[  \varsigma_{h}\right]  \geq\theta-\varepsilon
		;\varepsilon\rightarrow0.$Since $\mathcal{P}$ is weakly compact, there exists
		a subsequence $\left(  \mathbb{P}_{n_{k}}\right)  _{k\geq1}$ that converges
		weakly to some $\mathbb{P}\in\mathcal{P}$, hence
		\[
		\lim_{h\rightarrow0}\mathbb{E}^{\mathbb{P}}\left[  \varsigma_{h}\right]
		=\lim_{h\rightarrow0}\lim_{k\rightarrow\infty}\mathbb{E}^{\mathbb{P}_{n_{k}}%
		}\left[  \varsigma_{h}\right]  \geq\lim\inf_{k\rightarrow\infty}%
		\mathbb{E}^{\mathbb{P}_{n_{k}}}\left[  \varsigma_{h}\right]  \geq\theta.
		\]
		This contradicts (\ref{everyp}). This complete the proof.
	\end{proof}
	
	Since $u^{\ast}$ is optimal, then%
	\[
	J(u^{\ast})\leq J(u^{h})=J(u^{\ast})+\left.  h\frac{dJ(u^{h})}{dh}\right\vert
	_{h=0}+\circ(h)
	\]

	Thus a necessary condition for optimality is that
	\[
	\left.  \frac{dJ(u^{h})}{dh}\right\vert _{h=0}\geq0
	\]

	Note that under every $\mathbb{P}\in\mathcal{P}$, the following properties
	holds, because $b(t,x,u),$ $h(t,x,u)$, $\gamma(t,x,u)$ and $f(t,x_{t^{-}%
	},\theta,u)$ are sufficiently integrable%
	\begin{align}
		&  \frac{1}{h}%
		{\displaystyle\int\limits_{t}^{t+h}}
		\mathbb{E}^{\mathbb{P}}\left[  \left\vert k(s,x_{s},u_{s})-k(t,x_{t}%
		,u_{t})\right\vert ^{2}\right]  \text{ }\underrightarrow{h\rightarrow0}\text{
		}0\text{ }dt-a.e\label{prop1}\\
		&  \frac{1}{h}%
		{\displaystyle\int\limits_{t}^{t+h}}
		\mathbb{E}^{\mathbb{P}}\left[  \left\vert \gamma(s,x_{s},u_{s})-\gamma
		(t,x_{t},u_{t})\right\vert ^{2}\right]  \text{ }\underrightarrow{h\rightarrow
			0}\text{ }0\text{ }d\langle B\rangle_{t}-a.e \label{prop1'}%
	\end{align}

	\begin{equation}
		\frac{1}{h}%
		{\displaystyle\int\limits_{\Gamma}}
		{\displaystyle\int\limits_{t}^{t+h}}
		\mathbb{E}^{\mathbb{P}}\left[  \left\vert f(s,x_{s^{-}},\theta,u_{s}%
		)-f(t,x_{t^{-}},\theta,u_{t})\right\vert ^{2}\right]  \upsilon(d\theta)\text{
		}\underrightarrow{h\rightarrow0}\text{ }0\text{ }dt-a.e \label{prop2}%
	\end{equation}

	where $k$ stands for $b$ or $h.$
	
	\begin{lemma}
		\label{estimatelemma} Under assumptions (H$_{1}$)-(H$_{3}$), it holds that
		\[
		\lim_{h\rightarrow0}\widehat{E}\left[  \left\vert \frac{x_{t}^{h}-x_{t}^{\ast
		}}{h}-z_{t}\right\vert ^{2}\right]  =0.
		\]
		
	\end{lemma}
	
	\begin{proof}
		We proceed as in H.Ben Gherbal and B.Mezerdi \cite{HBG}, Let%
		\[
		y_{t}^{h}=\frac{x_{t}^{h}-x_{t}^{\ast}}{h}-z_{t}%
		\]

		Then, we have for $t\in\left[  t_{0;}t_{0}+h\right]  $%
		\[
		\left\{
		\begin{array}
			[c]{l}%
			\left.  dy_{t}^{h}=\frac{1}{h}\left[  b(t,x_{t}^{\ast}+h(y_{t}^{h}+z_{t}%
			),\nu)-b(t,x_{t}^{\ast},u_{t}^{\ast})-hb_{x}(t,x_{t}^{\ast},u_{t}^{\ast}%
			)z_{t}\right]  dt\right. \\
			\left.  +\frac{1}{h}\left[  \sigma(t,x_{t}^{\ast}+h(y_{t}^{h}+z_{t}%
			))-\sigma(t,x_{t}^{\ast})-h\sigma_{x}(t,x_{t}^{\ast})z_{t}\right]
			dB_{t}\right. \\
			\left.  \frac{1}{h}\left[  \gamma(t,x_{t}^{\ast}+h(y_{t}^{h}+z_{t}%
			),\nu)-\gamma(t,x_{t}^{\ast},u_{t}^{\ast})-h\gamma_{x}(t,x_{t}^{\ast}%
			,u_{t}^{\ast})z_{t}\right]  d\langle B\rangle_{t}\right. \\
			\left.  \frac{1}{h}%
			{\displaystyle\int\limits_{\Gamma}}
			\left[  f(t,x_{t^{-}}^{\ast}+h(y_{t^{-}}^{h}+z_{t^{-}}),\nu)-f(t,x_{t^{-}%
			}^{\ast},u_{t}^{\ast})-hf_{x}(t,x_{t^{-}}^{\ast},u_{t}^{\ast})z_{t^{-}%
			}\right]  \widetilde{N}{\small (dt,d\theta)}\right. \\
			\left.  y_{t_{0}}^{h}=-\left[  b(t_{0},x_{t_{0}}^{\ast},\nu)-b(t_{0},x_{t_{0}%
			}^{\ast},u_{t_{0}}^{\ast})\right]  .\right.
		\end{array}
		\right.
		\]

		Hence,
		\begin{align*}
			&  \left.  y_{t_{0}+h}^{h}=\frac{1}{h}%
			{\displaystyle\int\limits_{t_{0}}^{t_{0}+h}}
			\left[  b(t,x_{t}^{\ast}+h(y_{t}^{h}+z_{t}),\nu)-b(t,x_{t}^{\ast},\nu)\right]
			dt+\frac{1}{h}%
			{\displaystyle\int\limits_{t_{0}}^{t_{0}+h}}
			\left[  b(t,x_{t}^{\ast},\nu)-b(t,x_{t_{0}}^{\ast},\nu)\right]  dt\right. \\
			&  \left.  +\frac{1}{h}%
			{\displaystyle\int\limits_{t_{0}}^{t_{0}+h}}
			\left[  b(t,x_{t_{0}}^{\ast},\nu)-b(t_{0},x_{t_{0}}^{\ast},\nu)\right]
			dt+\frac{1}{h}%
			{\displaystyle\int\limits_{t_{0}}^{t_{0}+h}}
			\left[  b(t_{0},x_{t_{0}}^{\ast},u_{t_{0}}^{\ast})-b(t,x_{t}^{\ast}%
			,u_{t}^{\ast})\right]  dt\right. \\
			&  \left.  +\frac{1}{h}%
			{\displaystyle\int\limits_{t_{0}}^{t_{0}+h}}
			\left[  \sigma(t,x_{t}^{\ast}+h(y_{t}^{h}+z_{t}))-\sigma(t,x_{t}^{\ast
			})\right]  dB_{t}\right. \\
			&  \left.  \frac{1}{h}%
			{\displaystyle\int\limits_{t_{0}}^{t_{0}+h}}
			\left[  \gamma(t,x_{t}^{\ast}+h(y_{t}^{h}+z_{t}),\nu)-\gamma(t,x_{t}^{\ast
			},\nu)\right]  d\langle B\rangle_{t}+\frac{1}{h}%
			{\displaystyle\int\limits_{t_{0}}^{t_{0}+h}}
			\left[  \gamma(t,x_{t}^{\ast},\nu)-\gamma(t,x_{t_{0}}^{\ast},\nu)\right]
			d\langle B\rangle_{t}\right. \\
			&  \left.  +\frac{1}{h}%
			{\displaystyle\int\limits_{t_{0}}^{t_{0}+h}}
			\left[  \gamma(t,x_{t_{0}}^{\ast},\nu)-\gamma(t_{0},x_{t_{0}}^{\ast}%
			,\nu)\right]  d\langle B\rangle_{t}+\frac{1}{h}%
			{\displaystyle\int\limits_{t_{0}}^{t_{0}+h}}
			\left[  \gamma(t_{0},x_{t_{0}}^{\ast},u_{t_{0}}^{\ast})-\gamma(t,x_{t}^{\ast
			},u_{t}^{\ast})\right]  d\langle B\rangle_{t}\right. \\
			&  \left.  +\frac{1}{h}%
			{\displaystyle\int\limits_{t_{0}}^{t_{0}+h}}
			{\displaystyle\int\limits_{\Gamma}}
			\left[  f{\small (t,x_{t^{-}}^{\ast}+h(y_{t^{-}}^{h}+z_{t^{-}}),\theta,\nu
				)-}f{\small (t,x_{t^{-}}^{\ast},\theta,\nu)}\right]  \widetilde{N}%
			{\small (dt,d\theta)}\right. \\
			&  \left.  {\small +}\frac{1}{h}%
			{\displaystyle\int\limits_{t_{0}}^{t_{0}+h}}
			{\displaystyle\int\limits_{\Gamma}}
			\left[  f{\small (t,x_{t^{-}}^{\ast},\theta,\nu)-}f{\small (t,x_{t_{0}^{-}%
				}^{\ast},\theta,\nu)}\right]  \widetilde{N}{\small (dt,d\theta)}\right. \\
			&  \left.  {\small +}\frac{1}{h}%
			{\displaystyle\int\limits_{t_{0}}^{t_{0}+h}}
			{\displaystyle\int\limits_{\Gamma}}
			\left[  f{\small (t,x_{t_{0}^{-}}^{\ast},\theta,\nu)-}f{\small (t}%
			_{0}{\small ,x_{t_{0}^{-}}^{\ast},\theta,\nu)}\right]  \widetilde{N}%
			{\small (dt,d\theta)}\right. \\
			&  \left.  +\frac{1}{h}%
			{\displaystyle\int\limits_{t_{0}}^{t_{0}+h}}
			{\displaystyle\int\limits_{\Gamma}}
			\left[  f{\small (t}_{0}{\small ,x_{t_{0}^{-}}^{\ast},\theta,\nu
				)-}f{\small (t}_{0}{\small ,x_{t_{0}^{-}}^{\ast},\theta,u_{t_{0}}^{\ast}%
				)}\right]  \widetilde{N}{\small (dt,d\theta)}\right. \\
			&  \left.  {\small +}\frac{1}{h}%
			{\displaystyle\int\limits_{t_{0}}^{t_{0}+h}}
			{\displaystyle\int\limits_{\Gamma}}
			\left[  f{\small (t}_{0}{\small ,x_{t_{0}^{-}}^{\ast},\theta,u_{t_{0}}^{\ast
				})-}f{\small (t,x_{t^{-}}^{\ast},\theta,u_{t}^{\ast})}\right]  \widetilde{N}%
			{\small (dt,d\theta)}\right. \\
			&  \left.  -%
			{\displaystyle\int\limits_{t_{0}}^{t_{0}+h}}
			b_{x}(t,x_{t}^{\ast},u_{t}^{\ast})z_{t}dt-%
			{\displaystyle\int\limits_{t_{0}}^{t_{0}+h}}
			\gamma_{x}(t,x_{t}^{\ast},u_{t}^{\ast})z_{t}d\langle B\rangle_{t}\right. \\
			&  \left.  -%
			{\displaystyle\int\limits_{t_{0}}^{t_{0}+h}}
			\sigma_{x}(t,x_{t}^{\ast})z_{t}dB_{t}-%
			{\displaystyle\int\limits_{t_{0}}^{t_{0}+h}}
			{\displaystyle\int\limits_{\Gamma}}
			f_{x}{\small (t,x_{t^{-}}^{\ast},\theta,u_{t}^{\ast})}z_{t}\widetilde{N}%
			{\small (dt,d\theta)}\right.
		\end{align*}

		Then, under every $\mathbb{P}\in\mathcal{P}$, we have%
		\begin{equation}%
			\begin{array}
				[c]{l}%
				\left.  \mathbb{E}^{\mathbb{P}}\left\vert y_{t_{0}+h}^{h}\right\vert ^{2}\leq
				C\left[  \mathbb{E}^{\mathbb{P}}\sup\limits_{t_{0}\leq t\leq t_{0}%
					+h}\left\vert x_{t}^{h}-x_{t}^{\ast}\right\vert ^{2}+\sup\limits_{t_{0}\leq
					t\leq t_{0}+h}\mathbb{E}^{\mathbb{P}}\left\vert b(t,x_{t_{0}}^{\ast}%
				,\nu)-b(t_{0},x_{t_{0}}^{\ast},\nu)\right\vert ^{2}dt\right.  \right. \\
				\left.  +\frac{1}{h}\mathbb{E}^{\mathbb{P}}%
				{\displaystyle\int\limits_{t_{0}}^{t_{0}+h}}
				\left\vert b(t_{0},x_{t_{0}}^{\ast},u_{t_{0}}^{\ast})-b(t,x_{t}^{\ast}%
				,u_{t}^{\ast})\right\vert ^{2}dt+\mathbb{E}^{\mathbb{P}}\sup\limits_{t_{0}\leq
					t\leq t_{0}+h}\left\vert x_{t}^{\ast}-x_{t_{0}}^{\ast}\right\vert ^{2}\right.
				\\
				\left.  +\sup\limits_{t_{0}\leq t\leq t_{0}+h}\mathbb{E}^{\mathbb{P}%
				}\left\vert \gamma(t,x_{t_{0}}^{\ast},\nu)-\gamma(t_{0},x_{t_{0}}^{\ast}%
				,\nu)\right\vert ^{2}d\langle B\rangle_{t}+\frac{1}{h}\mathbb{E}^{\mathbb{P}}%
				{\displaystyle\int\limits_{t_{0}}^{t_{0}+h}}
				\left\vert \gamma(t_{0},x_{t_{0}}^{\ast},u_{t_{0}}^{\ast})-\gamma
				(t,x_{t}^{\ast},u_{t}^{\ast})\right\vert ^{2}d\langle B\rangle_{t}\right. \\
				\left.  +\mathbb{E}^{\mathbb{P}}%
				{\displaystyle\int\limits_{t_{0}}^{t_{0}+h}}
				{\displaystyle\int\limits_{\Gamma}}
				\left\vert \nu-u_{t_{0}}^{\ast}\right\vert ^{2}\upsilon(d\theta)dt+\mathbb{E}%
				^{\mathbb{P}}%
				{\displaystyle\int\limits_{t_{0}}^{t_{0}+h}}
				\left\vert z_{t}\right\vert ^{2}dt\right. \\
				+\left.  \sup\limits_{t_{0}\leq t\leq t_{0}+h}\mathbb{E}^{\mathbb{P}}%
				{\displaystyle\int\limits_{\Gamma}}
				\left\vert f{\small (t,x_{t_{0}^{-}}^{\ast},\theta,\nu)-}f{\small (t}%
				_{0}{\small ,x_{t_{0}^{-}}^{\ast},\theta,\nu)}\right\vert ^{2}\upsilon
				(d\theta)\right. \\
				\left.  +\frac{1}{h}\mathbb{E}^{\mathbb{P}}%
				{\displaystyle\int\limits_{t_{0}}^{t_{0}+h}}
				{\displaystyle\int\limits_{\Gamma}}
				\left\vert f{\small (t}_{0}{\small ,x_{t_{0}^{-}}^{\ast},\theta,u_{t_{0}%
					}^{\ast})-}f{\small (t,x_{t^{-}}^{\ast},\theta,u_{t}^{\ast})}\right\vert
				^{2}\upsilon(d\theta)dt\right]  .
			\end{array}
			\label{yini}%
		\end{equation}

		By lemma (\ref{xhxetoile}), and the properties (\ref{prop1}), (\ref{prop1'})
		and (\ref{prop2}), it is easy to see that for each $\mathbb{P}\in\mathcal{P}$,
		$\mathbb{E}^{\mathbb{P}}\left\vert y_{t_{0}+h}^{h}\right\vert ^{2}$ tends to
		$0$ as $h\rightarrow0.$
		
		Finally, we deduce that $\widehat{\mathbb{E}}\left\vert y_{t_{0}+h}%
		^{h}\right\vert ^{2}$ tends to $0$ as $h\rightarrow0$ by the same way as in
		the proof of lemma (\ref{xhxetoile}).
		
		For $t\in\left[  t_{0}+h;T\right]  ,$ we denote $x_{t}^{h,\lambda}=x_{t}%
		^{\ast}+\lambda h(y_{t}^{h}+z_{t}),$ then $y_{t}^{h}$ satisfies the following
		SDE%
		\begin{align*}
			&  \left.  dy_{t}^{h}=\frac{1}{h}\left[  b(t,x_{t}^{\ast}+h(y_{t}^{h}%
			+z_{t}),u_{t}^{\ast})-b(t,x_{t}^{\ast},u_{t}^{\ast})\right]  dt+\frac{1}%
			{h}\left[  \sigma(t,x_{t}^{\ast}+h(y_{t}^{h}+z_{t}))-\sigma(t,x_{t}^{\ast
			})\right]  dB_{t}\right. \\
			&  \left.  \frac{1}{h}\left[  \gamma(t,x_{t}^{\ast}+h(y_{t}^{h}+z_{t}%
			),u_{t}^{\ast})-\gamma(t,x_{t}^{\ast},u_{t}^{\ast})\right]  d\langle
			B\rangle_{t}\right. \\
			&  \left.  +\frac{1}{h}%
			{\displaystyle\int\limits_{\Gamma}}
			\left[  f{\small (t,x_{t^{-}}^{\ast}+h(y_{t^{-}}^{h}+z_{t^{-}}),\theta,u}%
			_{t}^{\ast}{\small )-}f{\small (t,x_{t^{-}}^{\ast},\theta,u_{t}^{\ast}%
				)}\right]  \widetilde{N}{\small (dt,d\theta)}\right. \\
			&  \left.  -b_{x}(t,x_{t}^{\ast},u_{t}^{\ast})z_{t}dt-\sigma_{x}(t,x_{t}%
			^{\ast})z_{t}dB_{t}-%
			{\displaystyle\int\limits_{\Gamma}}
			f_{x}{\small (t,x}_{t^{-}}^{\ast}{\small ,\theta,u}_{t}^{\ast}{\small )z_{t}%
			}\widetilde{N}{\small (dt,d\theta),}\right.
		\end{align*}

		then,
		\begin{align*}
			&  \left.  y_{t}^{h}=y_{t_{0}+h}^{h}+%
			{\displaystyle\int\limits_{t_{0}+h}^{t}}
			{\displaystyle\int\limits_{0}^{1}}
			b_{x}(s,x_{s}^{h,\lambda},u_{s}^{\ast})y_{s}^{h}d\lambda ds+%
			{\displaystyle\int\limits_{t_{0}+h}^{t}}
			{\displaystyle\int\limits_{0}^{1}}
			\sigma_{x}(s,x_{s}^{h,\lambda})y_{s}^{h}d\lambda dB_{s}\right. \\
			&  \left.
			{\displaystyle\int\limits_{t_{0}+h}^{t}}
			{\displaystyle\int\limits_{0}^{1}}
			\gamma_{x}(s,x_{s}^{h,\lambda},u_{s}^{\ast})y_{s}^{h}d\lambda d\langle
			B\rangle_{s}+%
			{\displaystyle\int\limits_{0}^{1}}
			{\displaystyle\int\limits_{t_{0}+h}^{t}}
			{\displaystyle\int\limits_{\Gamma}}
			f_{x}{\small (s,x_{s}^{h,\lambda},\theta,u_{s}^{\ast})}y_{s}^{h}%
			d\lambda\widetilde{N}{\small (ds,d\theta)+\rho}_{t}^{h},\right.
		\end{align*}

		where%
		\begin{align*}
			&  \left.  {\small \rho}_{t}^{h}=%
			{\displaystyle\int\limits_{t_{0}+h}^{t}}
			{\displaystyle\int\limits_{0}^{1}}
			b_{x}(s,x_{s}^{h,\lambda},u_{s}^{\ast})z_{s}d\lambda ds+%
			{\displaystyle\int\limits_{t_{0}+h}^{t}}
			{\displaystyle\int\limits_{0}^{1}}
			\sigma_{x}(s,x_{s}^{h,\lambda})z_{s}d\lambda dB_{s}\right. \\
			&  \left.  +%
			{\displaystyle\int\limits_{t_{0}+h}^{t}}
			{\displaystyle\int\limits_{0}^{1}}
			\gamma_{x}(s,x_{s}^{h,\lambda},u_{s}^{\ast})z_{s}d\lambda d\langle
			B\rangle_{s}+%
			{\displaystyle\int\limits_{t_{0}+h}^{t}}
			{\displaystyle\int\limits_{0}^{1}}
			{\displaystyle\int\limits_{\Gamma}}
			f_{x}{\small (s,x_{s}^{h,\lambda},\theta,u_{s}^{\ast})}z_{s}d\lambda
			\widetilde{N}{\small (ds,d\theta)}\right. \\
			&  \left.  {\small -}%
			{\displaystyle\int\limits_{t_{0}+h}^{t}}
			b_{x}(s,x_{s}^{\ast},u_{s}^{\ast})z_{s}ds{\small -}%
			{\displaystyle\int\limits_{t_{0}+h}^{t}}
			\gamma_{x}(s,x_{s}^{\ast},u_{s}^{\ast})z_{s}d\langle B\rangle_{s}-%
			{\displaystyle\int\limits_{t_{0}+h}^{t}}
			\sigma_{x}(s,x_{s}^{\ast})z_{s}dB_{s}\right. \\
			&  \left.  -%
			{\displaystyle\int\limits_{t_{0}+h}^{t}}
			{\displaystyle\int\limits_{\Gamma}}
			f_{x}{\small (s,x}_{s^{-}}^{\ast}{\small ,\theta,u}_{s}^{\ast}{\small )z_{s}%
			}\widetilde{N}{\small (ds,d\theta)}\right.
		\end{align*}

		hence, under every $\mathbb{P}\in\mathcal{P}$, we have
		\begin{align*}
			&  \left.  \mathbb{E}^{\mathbb{P}}\left\vert y_{t}^{h}\right\vert ^{2}%
			\leq\mathbb{E}^{\mathbb{P}}\left\vert y_{t_{0}+h}^{h}\right\vert
			^{2}+K\mathbb{E}^{\mathbb{P}}%
			{\displaystyle\int\limits_{t_{0}+h}^{t}}
			\left\vert
			{\displaystyle\int\limits_{0}^{1}}
			b_{x}(s,x_{s}^{h,\lambda},u_{s}^{h})y_{s}^{h}d\lambda\right\vert ^{2}ds\right.
			\\
			&  \left.  +K\mathbb{E}^{\mathbb{P}}%
			{\displaystyle\int\limits_{t_{0}+h}^{t}}
			\left\vert
			{\displaystyle\int\limits_{0}^{1}}
			\sigma_{x}(s,x_{s}^{h,\lambda})y_{s}^{h}d\lambda\right\vert ^{2}%
			ds+K\mathbb{E}^{\mathbb{P}}%
			{\displaystyle\int\limits_{t_{0}+h}^{t}}
			\left\vert
			{\displaystyle\int\limits_{0}^{1}}
			\gamma_{x}(s,x_{s}^{h,\lambda},u_{s}^{h})y_{s}^{h}d\lambda\right\vert
			^{2}d\langle B\rangle_{s}\right. \\
			&  \left.  +K\mathbb{E}^{\mathbb{P}}%
			{\displaystyle\int\limits_{t_{0}+h}^{t}}
			{\displaystyle\int\limits_{\Gamma}}
			\left\vert
			{\displaystyle\int\limits_{0}^{1}}
			f_{x}{\small (s,x_{s}^{h,\lambda},\theta,u_{s}^{h})}y_{s}^{h}d\lambda
			\right\vert ^{2}\upsilon{\small (d\theta)ds+K}\mathbb{E}^{\mathbb{P}%
			}\left\vert {\small \rho}_{t}^{h}\right\vert ^{2}\right.  .
		\end{align*}

		Since $b_{x},$ $\sigma_{x},$ $\gamma_{x}$ and $f_{x}$ are bounded, then
		\[
		\mathbb{E}^{\mathbb{P}}\left\vert y_{t}^{h}\right\vert ^{2}\leq\mathbb{E}%
		^{\mathbb{P}}\left\vert y_{t_{0}+h}^{h}\right\vert ^{2}+C\mathbb{E}%
		^{\mathbb{P}}%
		{\displaystyle\int\limits_{0}^{t}}
		\left\vert y_{s}^{h}\right\vert ^{2}ds+{\small K}\mathbb{E}^{\mathbb{P}%
		}\left\vert {\small \rho}_{t}^{h}\right\vert ^{2}%
		\]

		We conclude by the continuity of $b_{x},$ $\sigma_{x},$ $\gamma_{x}$ and
		$f_{x}$, and the dominated convergence that $\lim_{h\rightarrow0}{\small \rho
		}_{t}^{h}=0.$ Hence by the Gronwall lemma, and (\ref{yini}) we get
		\[
		\lim_{h\rightarrow0}\sup_{t_{0}+h\leq t\leq T}\mathbb{E}^{\mathbb{P}%
		}\left\vert y_{t}^{h}\right\vert ^{2}=0.
		\]

		Finally, we deduce that $\widehat{\mathbb{E}}\left\vert y_{t}^{h}\right\vert
		^{2}$ tends to $0$ as $h\rightarrow0$ by the same way as in the proof of lemma
		(\ref{xhxetoile}).
		
		The second estimate is proved in a similar way.
	\end{proof}
	
	Choose $t_{0}$ such that (\ref{prop1}), (\ref{prop1'}) and (\ref{prop2})
	holds, then we have
	
	\begin{corollary}
		\label{corDJ}Under assumptions (H$_{1}$)-(H$_{3}$), one has
		\begin{equation}
			0\leq\left.  \frac{dJ(u^{h})}{dh}\right\vert _{h=0}\leq\widehat{\mathbb{E}%
			}\left[  g_{x}(x_{T}^{\ast})z_{T}+%
			{\displaystyle\int\limits_{0}^{T}}
			h_{x}(t,x_{t}^{\ast},u_{t}^{\ast})z_{t}dt\right]  \label{dJ}%
		\end{equation}
		where the process $z$ is the solution of the linear SDE%
		\begin{equation}
			\left\{
			\begin{array}
				[c]{l}%
				\left.  dz_{t}=b_{x}(t,x_{t}^{\ast},u_{t}^{\ast})z_{t}dt+\sigma_{x}%
				(t,x_{t}^{\ast})z_{t}dB_{t}+\gamma_{x}(t,x_{t}^{\ast},u_{t}^{\ast}%
				)z_{t}d\langle B\rangle_{t}\right. \\
				\left.  +%
				{\displaystyle\int\limits_{\Gamma}}
				f_{x}(t,x_{t^{-}}^{\ast},\theta,u_{t}^{\ast})z_{t^{-}}\widetilde{N}%
				(dt,d\theta);\text{ }t_{0}\leq t\leq T\right. \\
				\left.  z_{t_{0}}=\left[  b(t_{0},x_{t_{0}}^{\ast},\nu)-b(t_{0},x_{t_{0}%
				}^{\ast},u_{t_{0}}^{\ast})\right]  \right.
			\end{array}
			\right.  \label{zt}%
		\end{equation}
		
	\end{corollary}
	
	We use the same notations as in the proof of Lemma (\ref{estimatelemma}), to
	prove this corollary.
	
	\begin{proof}
		[proof ]We have by the definition of $J$ that
		\[
		\frac{1}{h}\left[  J(u^{h})-J(u^{\ast})\right]  \leq\frac{1}{h}%
		\widehat{\mathbb{E}}\left[
		{\displaystyle\int\limits_{t_{0}}^{T}}
		h(t,x_{t}^{h},u_{t}^{h})+g(x_{T}^{h})\right]  -\widehat{\mathbb{E}}\left[
		{\displaystyle\int\limits_{t_{0}}^{T}}
		h(t,x_{t}^{\ast},u_{t}^{\ast})+g(x_{T}^{\ast})\right]  dt
		\]

		then,%
		\[%
		\begin{array}
			[c]{l}%
			0\leq\frac{1}{h}\left[  J(u^{h})-J(u^{\ast})\right]  \leq\widehat{\mathbb{E}%
			}\left[
			{\displaystyle\int\limits_{0}^{1}}
			g_{x}(x_{T}^{h,\lambda})z_{T}d\lambda+\frac{1}{h}%
			{\displaystyle\int\limits_{0}^{T}}
			{\displaystyle\int\limits_{0}^{1}}
			h_{x}(t,x_{t}^{h,\lambda},u_{t}^{h})z_{tt}^{h}d\lambda dt\right. \\
			\left.  +%
			{\displaystyle\int\limits_{0}^{T}}
			{\displaystyle\int\limits_{0}^{1}}
			h_{u}(t,x_{t}^{\ast},u_{t}^{h,\lambda})u_{t}d\lambda dt\right]
		\end{array}
		.
		\]

		From Lemma (\ref{estimatelemma}), we obtain (\ref{dJ}) by letting $h$ tend to
		$0.$\bigskip
	\end{proof}
	
	Let us introduce the adjoint process, which is a $G$-backward stochastic
	differential equation (G-BSDE in short). We proceed as in \cite{bensoussan},
	\cite{base}.and \cite{HBG}.
	
	By the integration by parts formula, we can see that the solution of $dz_{t}$
	is given by $z_{t}=\varphi_{t}\eta_{t}$ where
	\[
	\left\{
	\begin{array}
		[c]{l}%
		\left.
		\begin{array}
			[c]{c}%
			d\varphi(t,\tau)=b_{x}(t,x_{t}^{\ast},u_{t}^{\ast})\varphi(t,\tau
			)dt+\sigma_{x}(t,x_{t}^{\ast})\varphi(t,\tau)dB_{t}\\
			+%
			{\displaystyle\int\limits_{\Gamma}}
			f_{x}(t,x_{t^{-}}^{\ast},\theta,u_{t}^{\ast})\varphi(t^{-},\tau)\widetilde{N}%
			(dt,d\theta)+\gamma_{x}(t,x_{t}^{\ast},u_{t}^{\ast})d\langle B\rangle_{t}%
		\end{array}
		\right.  \text{ \ \ }0\leq\tau\leq t\leq T,\\
		\varphi(\tau,\tau)=I_{d}%
	\end{array}
	\right.
	\]

	and
	\[
	\left\{
	\begin{array}
		[c]{l}%
		\left.
		\begin{array}
			[c]{l}%
			d\eta_{t}=\psi_{t}\left\{  b_{u}(t,x_{t}^{\ast},u_{t}^{\ast})u_{t}-%
			{\displaystyle\int\limits_{\Gamma}}
			f_{u}(t,x_{t^{-}}^{\ast},\theta,u_{t}^{\ast})u_{t}\upsilon(d\theta)\right\}
			dt\\
			-\psi_{t^{-}}%
			{\displaystyle\int\limits_{\Gamma}}
			\left(  f_{x}(t,x_{t^{-}}^{\ast},\theta,u_{t}^{\ast})+I_{d}\right)  ^{-1}%
			f_{u}(t,x_{t^{-}}^{\ast},\theta,u_{t}^{\ast})u_{t}N(dt,d\theta)\\
			+\psi_{t}\gamma_{u}(t,x_{t}^{\ast},u_{t}^{\ast})u_{t}d\langle B\rangle_{t}%
		\end{array}
		\right.  \text{ \ }\\
		\eta_{0}=0,
	\end{array}
	\right.
	\]

	with $\psi_{t}$ is the inverse of $\varphi$ satisfying suitable integrability
	conditions, and it is the solution of the following equation%
	\[
	\left\{
	\begin{array}
		[c]{l}%
		\left.
		\begin{array}
			[c]{l}%
			\begin{array}
				[c]{l}%
				\begin{array}
					[c]{l}%
					d\psi(t,\tau)=\left\{  \sigma_{x}(t,x_{t}^{\ast})\psi(t,\tau)\sigma
					_{x}(t,x_{t}^{\ast})-b_{x}(t,x_{t}^{\ast},u_{t}^{\ast})\psi(t,\tau)\right. \\
					\left.  -%
					{\displaystyle\int\limits_{\Gamma}}
					f_{x}(t,x_{t^{-}}^{\ast},\theta,u_{t}^{\ast})\psi(t^{-},\tau)\upsilon
					(d\theta)\right\}  dt
				\end{array}
				\\
				-\sigma_{x}(t,x_{t}^{\ast})\psi(t,\tau)dB_{t}-\gamma_{x}(t,x_{t}^{\ast}%
				,u_{t}^{\ast})d\langle B\rangle_{t}%
			\end{array}
			\\
			-\psi(t^{-},\tau)%
			{\displaystyle\int\limits_{\Gamma}}
			\left(  f_{x}(t,x_{t^{-}}^{\ast},\theta,u_{t}^{\ast})+I_{d}\right)  ^{-1}%
			f_{x}(t,x_{t^{-}}^{\ast},\theta,u_{t}^{\ast})N(dt,d\theta)
		\end{array}
		\right.  \text{ \ \ }0\leq\tau\leq t\leq T\\
		\psi(\tau,\tau)=I_{d}.
	\end{array}
	\right.
	\]

	\begin{remark}
		\begin{enumerate}
			\item From It\^o's formula, we can easily check that $d\left(  \varphi
			(t,\tau)\psi(t,\tau)\right)  =0,$ and $\varphi(\tau,\tau)\psi(\tau,\tau
			)=I_{d}.$
			
			\item If $\tau=0,$ we simply write $\varphi(t,0)=\varphi_{t}$ and
			$\psi(t,0)=\psi_{t}.$
		\end{enumerate}
	\end{remark}
	
	Then the equality (\ref{dJ}) will become%
	\begin{equation}
		\left.
		\begin{array}
			[c]{l}%
			\left.  \frac{dJ(u^{h})}{dh}\right\vert _{h=0}=\widehat{\mathbb{E}}\left[
			{\displaystyle\int\limits_{0}^{T}}
			\left\{  h_{x}(t,x_{t}^{\ast},u_{t}^{\ast})\varphi_{t}\eta_{t}+h_{u}%
			(t,x_{t}^{\ast},u_{t}^{\ast})u_{t}\right\}  dt\right. \\
			\left.  +g_{x}(x_{T}^{\ast})\varphi_{T}\eta_{T}\right]
		\end{array}
		\right.  \label{***}%
	\end{equation}

	Set
	\begin{align*}
		X  &  =%
		{\displaystyle\int\limits_{0}^{T}}
		h_{x}(t,x_{t}^{\ast},u_{t}^{\ast})\varphi_{t}^{\ast}dt+g_{x}(x_{T}^{\ast
		})\varphi_{T}^{\ast}\\
		y_{t}  &  =\widehat{\mathbb{E}}\left[  X\diagup\mathcal{F}_{t}\right]  -%
		{\displaystyle\int\limits_{0}^{t}}
		h_{x}(s,x_{s}^{\ast},u_{s}^{\ast})\varphi_{s}^{\ast}ds+%
		{\displaystyle\int\limits_{0}^{t}}
		dk_{s}%
	\end{align*}

	then, we have
	\begin{equation}
		\left.
		\begin{array}
			[c]{l}%
			y_{T}=\widehat{\mathbb{E}}\left[  X\diagup\mathcal{F}_{t}\right]  -%
			{\displaystyle\int\limits_{0}^{T}}
			h_{x}(s,x_{s}^{\ast},u_{s}^{\ast})\varphi_{s}^{\ast}ds+%
			{\displaystyle\int\limits_{0}^{T}}
			dk_{t}\\
			=X-%
			{\displaystyle\int\limits_{0}^{T}}
			h_{x}(s,x_{s}^{\ast},u_{s}^{\ast})\varphi_{s}^{\ast}ds=g_{x}(x_{T}^{\ast
			})\varphi_{T}^{\ast}+%
			{\displaystyle\int\limits_{0}^{T}}
			dk_{t}%
		\end{array}
		\right.  \label{yT}%
	\end{equation}

	replacing (\ref{yT}) in (\ref{***}), we obtain%
	\begin{equation}
		\left.  \frac{dJ(u^{h})}{dh}\right\vert _{h=0}=\widehat{\mathbb{E}}\left[
		{\displaystyle\int\limits_{0}^{T}}
		\left\{  h_{x}(t,x_{t}^{\ast},u_{t}^{\ast})\varphi_{t}^{\ast}\eta_{t}%
		+h_{u}(t,x_{t}^{\ast},u_{t}^{\ast})u_{t}\right\}  dt+y_{T}\eta_{T}\right]  .
		\label{dJ+++}%
	\end{equation}

	By the It\^o representation theorem of a $G$-martingale (see \cite{peng-pre}),
	there exist two processes $Q\in M_{G}^{2}\left(  0,T\right)  ,S_{s}\in S(d)$
	and $R\in\mathcal{L}_{G}^{2}\left(  0,T\right)  $ satisfying
	\[
	\widehat{\mathbb{E}}\left[  X\diagup\mathcal{F}_{t}\right]
	=\widehat{\mathbb{E}}\left[  X\right]  +%
	{\displaystyle\int\limits_{0}^{t}}
	Q_{s}dB_{s}+%
	{\displaystyle\int\limits_{0}^{t}}
	\varphi_{s}^{\ast}S_{s}d\langle B\rangle_{s}-2%
	{\displaystyle\int\limits_{0}^{t}}
	\varphi_{s}^{\ast}G(S_{s})ds+%
	{\displaystyle\int\limits_{0}^{t}}
	{\displaystyle\int\limits_{\Gamma}}
	R_{s}(\theta)\widetilde{N}(ds,d\theta),
	\]
	where $G$ the generator $G:S(d)\rightarrow%
	\mathbb{R}
	$ satisfying the uniformly elliptic condition, i.e., there exists a $\beta>0$
	such that, for each $A,\overline{A}\in S(d)$ with $A\geq\overline{A},$
	\[
	G(A)-G(\overline{A})\geq\beta tr[A-\overline{A}].
	\]
	Hence,
	\[
	\left.
	\begin{array}
		[c]{l}%
		y_{t}=\widehat{\mathbb{E}}\left[  X\right]  -%
		{\displaystyle\int\limits_{0}^{t}}
		\left(  h_{x}(s,x_{s}^{\ast},u_{s}^{\ast})\varphi_{s}+2\varphi_{s}^{\ast
		}G(S_{s})\right)  ds+%
		{\displaystyle\int\limits_{0}^{t}}
		Q_{s}dB_{s}\\
		+%
		{\displaystyle\int\limits_{0}^{t}}
		{\displaystyle\int\limits_{\Gamma}}
		R_{s}(\theta)\widetilde{N}(ds,d\theta)+%
		{\displaystyle\int\limits_{0}^{t}}
		dk_{s}+%
		{\displaystyle\int\limits_{0}^{t}}
		\varphi_{s}^{\ast}S_{s}d\langle B\rangle_{s}%
	\end{array}
	\right.
	\]

	Now, let us calculate $\widehat{\mathbb{E}}\left[  y_{T}\eta_{T}\right]  ,$ we
	have
	\[
	dy_{t}=-\left(  h_{x}(s,x_{s}^{\ast},u_{s}^{\ast})\varphi_{s}+2\varphi
	_{s}^{\ast}G(S_{s})\right)  dt+Q_{t}dB_{t}+%
	{\displaystyle\int\limits_{\Gamma}}
	R_{t}(\theta)\widetilde{N}(dt,d\theta)+dk_{t}+\varphi_{s}^{\ast}S_{t}d\langle
	B\rangle_{t},
	\]

	by the integration by parts formula we get
	\[
	\left.
	\begin{array}
		[c]{l}%
		d(y_{t}\eta_{t})=y_{t}\psi_{t}\left[  b_{u}(t,x_{t}^{\ast},u_{t}^{\ast})u_{t}-%
		{\displaystyle\int\limits_{\Gamma}}
		f_{u}(t,x_{s}^{\ast},\theta,u_{s}^{\ast})u_{t}\upsilon(d\theta)\right]  dt\\
		-y_{t}\psi_{t^{-}}%
		{\displaystyle\int\limits_{\Gamma}}
		\left(  f_{x}+Id\right)  ^{-1}f_{u}u_{t}N(dt,d\theta)-\left(  \eta_{t}%
		\varphi_{t}^{\ast}h_{x}+2\eta_{t}\varphi_{s}^{\ast}G(S_{t})\right)  dt\\
		+\eta_{t}Q_{t}dB_{t}+%
		{\displaystyle\int\limits_{\Gamma}}
		\eta_{t}R_{t}(\theta)\widetilde{N}(dt,d\theta)+\left\{  y_{t}\psi_{t}%
		\gamma_{u}(t,x_{t}^{\ast},u_{t}^{\ast})u_{t}+q_{t}\sigma_{x}\eta_{t}%
		\varphi_{t}^{\ast}+\eta_{t}\varphi_{s}^{\ast}S_{t}\right\}  d\langle
		B\rangle_{t}\\
		+%
		{\displaystyle\int\limits_{\Gamma}}
		R_{t}(\theta)\psi_{t}\left(  f_{x}+Id\right)  ^{-1}f_{u}u_{t}\upsilon
		(d\theta)dt+\eta_{t}\varphi_{t}^{\ast}dk_{t}.
	\end{array}
	\right.
	\]

	If we define the adjoint process by : $p_{t}=y_{t}\psi_{t},$ then
	\[
	\left.
	\begin{array}
		[c]{l}%
		d(y_{t}\eta_{t})=p_{t}b_{u}udt-p_{t}%
		{\displaystyle\int\limits_{\Gamma}}
		f_{u}u\upsilon(d\theta)dt-p_{t}%
		{\displaystyle\int\limits_{\Gamma}}
		\left(  f_{x}+Id\right)  ^{-1}f_{u}u_{t}\widetilde{N}(dt,d\theta)\\
		-p_{t}%
		{\displaystyle\int\limits_{\Gamma}}
		\left(  f_{x}+Id\right)  ^{-1}f_{u}u_{t}\upsilon(d\theta)dt-\left(  \eta
		_{t}\varphi_{t}^{\ast}h_{x}+2\eta_{t}\varphi_{s}^{\ast}G(S_{t})\right)
		dt+\eta_{t}Q_{t}dB_{t}\\
		+\left\{  p_{t}\gamma_{u}(t,x_{t}^{\ast},u_{t}^{\ast})u_{t}+q_{t}\sigma
		_{x}\eta_{t}\varphi_{t}^{\ast}+\eta_{t}\varphi_{s}^{\ast}S_{t}\right\}
		d\langle B\rangle_{t}++\eta_{t}\varphi_{t}^{\ast}dk_{t}\\
		+%
		{\displaystyle\int\limits_{\Gamma}}
		\eta_{t}R_{t}(\theta)\widetilde{N}(dt,d\theta)+%
		{\displaystyle\int\limits_{\Gamma}}
		R_{t}(\theta)\psi_{t}\left(  f_{x}+Id\right)  ^{-1}f_{u}u_{t}\upsilon
		(d\theta)dt,
	\end{array}
	\right.
	\]

	Hence%
	\begin{align*}
		y_{T}\eta_{T}  &  =%
		{\displaystyle\int\limits_{0}^{T}}
		p_{t}b_{u}udt-%
		{\displaystyle\int\limits_{0}^{T}}
		{\displaystyle\int\limits_{\Gamma}}
		p_{t}f_{u}u\upsilon(d\theta)dt-%
		{\displaystyle\int\limits_{0}^{T}}
		{\displaystyle\int\limits_{\Gamma}}
		p_{t}\left(  f_{x}+Id\right)  ^{-1}f_{u}u_{t}\widetilde{N}(dt,d\theta)\\
		&  -%
		{\displaystyle\int\limits_{0}^{T}}
		{\displaystyle\int\limits_{\Gamma}}
		p_{t}\left(  f_{x}+Id\right)  ^{-1}f_{u}u_{t}\upsilon(d\theta)dt-%
		{\displaystyle\int\limits_{0}^{T}}
		\left(  \eta_{t}\varphi_{t}^{\ast}h_{x}+2\eta_{t}\varphi_{s}^{\ast}%
		G(S_{t})\right)  dt+%
		{\displaystyle\int\limits_{0}^{T}}
		\eta_{t}Q_{t}dB_{t}\\
		&  +%
		{\displaystyle\int\limits_{0}^{T}}
		\left\{  p_{t}\gamma_{u}(t,x_{t}^{\ast},u_{t}^{\ast})u_{t}+q_{t}\sigma_{x}%
		\eta_{t}\varphi_{t}^{\ast}+\eta_{t}S_{t}\right\}  d\langle B\rangle_{t}+%
		{\displaystyle\int\limits_{0}^{T}}
		\eta_{t}\varphi_{t}^{\ast}dk_{t}\\
		&  +%
		{\displaystyle\int\limits_{0}^{T}}
		{\displaystyle\int\limits_{\Gamma}}
		\eta_{t}R_{t}(\theta)\widetilde{N}(dt,d\theta)+%
		{\displaystyle\int\limits_{0}^{T}}
		{\displaystyle\int\limits_{\Gamma}}
		R_{t}(\theta)\psi_{t}\left(  f_{x}+Id\right)  ^{-1}f_{u}u_{t}\upsilon
		(d\theta)dt,
	\end{align*}

	take the G-expectation, we obtain%
	\[%
	\begin{array}
		[c]{l}%
		\widehat{\mathbb{E}}\left[  y_{T}\eta_{T}\right]  =\widehat{\mathbb{E}}\left[
		%
		{\displaystyle\int\limits_{0}^{T}}
		p_{t}b_{u}udt+%
		{\displaystyle\int\limits_{0}^{T}}
		\left[
		{\displaystyle\int\limits_{\Gamma}}
		R_{t}(\theta)\psi_{t}\left(  f_{x}+Id\right)  ^{-1}-p_{t}\left(  \left(
		f_{x}+Id\right)  ^{-1}+Id\right)  \right]  f_{u}u_{t}\upsilon(d\theta
		)dt\right. \\
		\left.  +%
		{\displaystyle\int\limits_{0}^{T}}
		\left\{  p_{t}\gamma_{u}(t,x_{t}^{\ast},u_{t}^{\ast})u_{t}+q_{t}\sigma_{x}%
		\eta_{t}\varphi_{t}^{\ast}+\eta_{t}\varphi_{s}^{\ast}S_{t}\right\}  d\langle
		B\rangle_{t}+%
		{\displaystyle\int\limits_{0}^{T}}
		\eta_{t}\varphi_{t}^{\ast}dk_{t}-%
		{\displaystyle\int\limits_{0}^{T}}
		\left(  \eta_{t}\varphi_{t}^{\ast}h_{x}+2\eta_{t}\varphi_{s}^{\ast}%
		G(S_{t})\right)  dt\right]
	\end{array}
	\]

	We define the adjoint process $r$ by%
	\[
	r_{t}(\theta)=R_{t}(\theta)\psi_{t}\left(  f_{x}+Id\right)  ^{-1}-p_{t}\left(
	\left(  f_{x}+Id\right)  ^{-1}+Id\right)  ,
	\]

	hence,
	\[
	\left.
	\begin{array}
		[c]{c}%
		\widehat{\mathbb{E}}\left[  y_{T}\eta_{T}\right]  =\widehat{\mathbb{E}}\left[
		%
		{\displaystyle\int\limits_{0}^{T}}
		\left\{  p_{t}b_{u}u+%
		{\displaystyle\int\limits_{\Gamma}}
		r_{t}(\theta)f_{u}u_{t}\upsilon(d\theta)\right\}  dt-%
		{\displaystyle\int\limits_{0}^{T}}
		\left(  \eta_{t}\varphi_{t}^{\ast}h_{x}+2\eta_{t}\varphi_{s}^{\ast}%
		G(S_{t})\right)  dt\right. \\
		\multicolumn{1}{l}{\left.  +%
			{\displaystyle\int\limits_{0}^{T}}
			\left\{  \gamma_{u}(t,x_{t}^{\ast},u_{t}^{\ast})u_{t}+q_{t}\sigma_{x}\eta
			_{t}\varphi_{t}^{\ast}+\eta_{t}\varphi_{s}^{\ast}S_{t}\right\}  d\langle
			B\rangle_{t}+%
			{\displaystyle\int\limits_{0}^{T}}
			\eta_{t}\varphi_{t}^{\ast}dk_{t}\right]  .}%
	\end{array}
	\right.
	\]

	By the replacing in (\ref{dJ+++}), we get
	\begin{equation}%
		\begin{array}
			[c]{l}%
			\left.  \frac{dJ(u^{h})}{dh}\right\vert _{h=0}=E\left[
			{\displaystyle\int\limits_{0}^{T}}
			\left\{  h_{u}(s,x_{s}^{\ast},u_{s}^{\ast})+p_{s}b_{u}(s,x_{s}^{\ast}%
			,u_{s}^{\ast})+%
			{\displaystyle\int\limits_{\Gamma}}
			r_{s}(\theta)f_{u}(s,x_{s}^{\ast},\theta,u_{s}^{\ast})\upsilon(d\theta
			)\right\}  u_{s}ds\right. \\
			\left.  +%
			{\displaystyle\int\limits_{0}^{T}}
			\left\{  p_{t}\gamma_{u}(t,x_{t}^{\ast},u_{t}^{\ast})u_{t}+q_{t}\sigma_{x}%
			\eta_{t}\varphi_{t}^{\ast}+\eta_{t}\varphi_{s}^{\ast}S_{t}\right\}  d\langle
			B\rangle_{t}-%
			{\displaystyle\int\limits_{0}^{T}}
			2\eta_{t}\varphi_{s}^{\ast}G(S_{t})dt+%
			{\displaystyle\int\limits_{0}^{T}}
			\eta_{t}\varphi_{t}^{\ast}dk_{t}\right]  \geq0.
		\end{array}
		\label{dj2}%
	\end{equation}

	Finally, based on the remark $5.2$ in \cite{base} if we assume that in
	equation (\ref{dj2}) $k=0$ $q.s$ and we define the Hamiltonian $H$ from
	$\left[  0;T\right]  \times%
	\mathbb{R}
	^{n}\times A\times%
	\mathbb{R}
	^{n}\times%
	\mathbb{R}
	^{n\times m}\times L_{m}^{2}$ into $%
	\mathbb{R}
	$ by
	\begin{equation}
		\left.
		\begin{array}
			[c]{l}%
			H(t,x,u,p,q,r(.))=h(t,x_{t},u_{t})+pb(t,x_{t},u_{t})\\
			+q\sigma(t,x_{t})+%
			{\displaystyle\int\limits_{\Gamma}}
			r_{t}(\theta)f(s,x_{t},\theta,u_{t})\upsilon(d\theta).
		\end{array}
		\right.  \label{H}%
	\end{equation}

	and%
	\[
	F(t,x,u,p,q,r(.))=%
	{\displaystyle\int\limits_{0}^{T}}
	\left\{  p_{t}\gamma_{u}(t,x_{t}^{\ast},u_{t}^{\ast})u_{t}+q_{t}\sigma_{x}%
	\eta_{t}\varphi_{t}^{\ast}+\eta_{t}\varphi_{s}^{\ast}S_{t}\right\}  d\langle
	B\rangle_{t}-%
	{\displaystyle\int\limits_{0}^{T}}
	2\eta_{t}\varphi_{s}^{\ast}G(S_{t})dt
	\]

	we get from (\ref{dj2}) the next theorem, which is the result of this subsection.
	
	\begin{theorem}
		[maximum principle for strict control]Let $u^{\ast}$ be the optimal strict
		control minimizing the cost $J$ $(.)$ over $U$, and denote by $x^{\ast}$ the
		corresponding optimal trajectory. Then there exists a unique triple of square
		integrable adapted processes $(p,q,r)$ which is the unique solution of the
		backward G-SDE%
		\begin{equation}
			\left\{
			\begin{array}
				[c]{l}%
				\begin{array}
					[c]{c}%
					\left.  dp_{t}=-\left\{  h_{x}(t,x_{t}^{\ast},u_{t}^{\ast})+p_{t}b_{x}%
					(t,x_{t}^{\ast},u_{t}^{\ast})+%
					{\displaystyle\int\limits_{\Gamma}}
					r_{t}(\theta)f(t,x_{t^{-}}^{\ast},\theta,u_{t}^{\ast})\upsilon(d\theta
					)\right\}  dt\right. \\
					\left.  -\left\{  \gamma_{x}(t,x_{t}^{\ast},u_{t}^{\ast})p_{t}+q_{t}\sigma
					_{x}(t,x_{t}^{\ast})\right\}  d\langle B\rangle_{t}+q_{t}dB_{t}+%
					{\displaystyle\int\limits_{\Gamma}}
					r_{t}(\theta)\widetilde{N}(dt,d\theta)+dk_{t}\right.
				\end{array}
				\\
				p_{T}=g_{x}(x_{T}^{\ast}),\text{ }k_{0}=0
			\end{array}
			\right.
		\end{equation}
		such that, if we assume that $b=0$ and $h=0,$ then for all $\nu\in U$ the
		following inequality holds
		\[
		\widehat{\mathbb{E}}\left[  H(t,x_{t}^{\ast},\nu,p_{t})-H(t,x_{t}^{\ast}%
		,u_{t}^{\ast},p_{t})+F(t,x_{t}^{\ast},u_{t}^{\ast},p,q,r(.))\right]
		\geq0.dt-a.e.
		\]
		where the Hamiltonian $H$ is defined by (\ref{H}).
	\end{theorem}
	
	\subsection{The maximum principle for near optimal controls}
	
	In this subsection, we establish necessary conditions of near optimality
	satisfied by a sequence of nearly optimal strict controls. This result is
	based on Ekeland's variational principle, which is given by the following lemma
	
	\begin{lemma}
		\label{Eke} [Ekeland's variational principle]\label{Ekeland}Let $(E,d)$ be a
		complete metric space and $f:E\rightarrow\overline{%
			\mathbb{R}
		}$ be lower semicontinuous and bounded from below. Given $\varepsilon>0$,
		suppose $u^{\varepsilon}\in E$ satisfies $f(u^{\varepsilon})$ $\leq
		\inf(f)+\varepsilon.$ Then for any $\lambda>0,$ there exists $\nu\in E$ such that
		
		\begin{itemize}
			\item $f(\nu)$ $\leq f(u^{\varepsilon})$
			
			\item $d(u^{\varepsilon},\nu)\leq\lambda$
			
			\item $f(\nu)\leq f(\omega)+\frac{\varepsilon}{\lambda}d(\omega,\nu)$ for all
			$\omega\neq\nu.$
		\end{itemize}
	\end{lemma}
	
	To apply Ekeland's variational principle, we have to endow the set $U$ of
	strict controls with an appropriate metric. For any $u$ and $\nu\in U,$ we set%
	\[
	d(u,\nu)=\mathbb{P}\otimes dt\left\{  (\omega,t)\in\Omega\times\left[
	0;T\right]  ;u(t,\omega)\neq\nu(t,\omega)\right\}
	\]

	where $\mathbb{P}\otimes dt$ is the product measure of $\mathbb{P}$ with the
	Lebesgue measure $dt.$
	
	\begin{remark}
		It is easy to see that $(U,d)$ is a complete metric space, and it well known
		that the cost functional $J$ is continuous from $U$ into $%
		\mathbb{R}
		$. For more detail see \cite{Mez}.
	\end{remark}
	
	Now, let $\mu^{\ast}\in\mathcal{R}$ be an optimal relaxed control and denote
	by $x^{\mu^{\ast}}$ the trajectory of the system controlled by $\mu^{\ast}.$
	From Lemma (\ref{chattlemma}), there exists a sequence $(u^{n})$ of strict
	controls such that
	\[
	\mu_{t}^{n}(da)dt=\delta_{u_{t}^{n}}(da)dt\longrightarrow\mu_{t}^{\ast
	}(da)dt\text{ \ \ \ quasi-surely}%
	\]

	and for every $\mathbb{P}\in\mathcal{P}$
	\[
	\lim_{n\rightarrow\infty}\mathbb{E}^{\mathbb{P}}\left[  \left\vert x_{t}%
	^{n}-x_{t}^{\mu^{\ast}}\right\vert ^{2}\right]  =0
	\]

	where $x^{n}$ is the solution of (\ref{relaxedSDE}) corresponding to $\mu
	^{n}.$
	
	According to the optimality of $\mu^{\ast}$ and lemma (\ref{Ekeland}), there
	exists a sequence $(\varepsilon_{n})$ of positive numbers with $\lim
	_{n\rightarrow\infty}\varepsilon_{n}=0$ such that
	\[
	J(u^{n})=J(\mu^{n})\leq J(\mu^{\ast})+\varepsilon_{n}=\inf_{u\in
		U}J(u)+\varepsilon_{n}%
	\]

	a suitable version of Lemma (\ref{Ekeland}) implies that, given any
	$\varepsilon_{n}>0,$ there exists $u^{n}\in U$ such that
	\begin{equation}
		J(u^{n})\leq J(u)+\varepsilon_{n}d(u^{n},u)\text{, }\forall u\in U
		\label{Jepsilon}%
	\end{equation}

	Let us define the perturbation
	\[
	u^{n,h}=\left\{
	\begin{array}
		[c]{l}%
		\nu\text{ \ \ \ if \ }t\in\left[  t_{0};t_{0}+h\right] \\
		u^{n}\text{ \ \ \ \ \ \ otherwise }%
	\end{array}
	\right.
	\]

	From (\ref{Jepsilon}) we have
	\[
	0\leq J(u^{n,h})-J(u^{n})+\varepsilon_{n}d(u^{n,h},u^{n})
	\]

	Using the definition of $d$ it holds that
	\begin{equation}
		0\leq J(u^{n,h})-J(u^{n})+\varepsilon_{n}Ch \label{inequality}%
	\end{equation}

	where $C$ is a positive constant.
	
	Now, we can introduce the next theorem which is the main result of this section.
	
	\begin{theorem}
		For each $\varepsilon_{n}>0,$ there exists $(u^{n})\in U$ such that there
		exists a unique triple of square integrable adapted processes $(p^{n}%
		,q^{n},r^{n})$ which is the solution of the backward SDE%
		\begin{equation}
			\left\{
			\begin{array}
				[c]{l}%
				\left.
				\begin{array}
					[c]{l}%
					dp_{t}^{n}=-\left\{  h_{x}(t,x_{t}^{n},u_{t}^{n})+p_{t}^{n}b_{x}(t,x_{t}%
					^{n},u_{t}^{n})+%
					{\displaystyle\int\limits_{\Gamma}}
					r_{t}^{n}(\theta)f(t,x_{t^{-}}^{n},\theta,u_{t}^{n})\upsilon(d\theta)\right\}
					dt\\
					\left.  -\left\{  \gamma_{x}(t,x_{t}^{n},u_{t}^{n})p_{t}+q_{t}^{n}\sigma
					_{x}(t,x_{t}^{n})\right\}  d\langle B\rangle_{t}+q_{t}^{n}dB_{t}+%
					{\displaystyle\int\limits_{\Gamma}}
					r_{t}^{n}(\theta)\widetilde{N}(dt,d\theta)+dk_{t}^{n}\right.
				\end{array}
				\right. \\
				p_{T}^{n}=g_{x}(x_{T}^{n}),\text{ }k_{0}^{n}=0
			\end{array}
			\right.  \label{pnn}%
		\end{equation}
		such that, if we assume that in equation (\ref{pnn}) $\forall n,$
		$h_{x}(t,x_{t}^{n},u_{t}^{n})=0,b_{x}(t,x_{t}^{n},u_{t}^{n})=0,$ then for all
		$\nu\in U$%
		\begin{equation}
			\left.
			\begin{array}
				[c]{l}%
				\widehat{\mathbb{E}}\left[  H(t,x_{t}^{n},\nu,p_{t}^{n})-H(t,x_{t}^{n}%
				,u_{t}^{n},p_{t}^{n})+G^{n}(t,x_{t}^{\ast},u_{t}^{\ast},p,q,r(.))\right] \\
				+C\varepsilon_{n}\geq0.
			\end{array}
			\right.  dt-a.e. \label{main2}%
		\end{equation}
		where $C$ is a positive constant.
	\end{theorem}
	
	\begin{proof}
		From the inequality (\ref{inequality}), we use the same method as in the
		previous subsection, we obtain (\ref{main2}).
	\end{proof}
	
	\subsection{The relaxed stochastic maximum principle}
	
	Now, we can introduce the next theorem, which is the main result of this section
	
	\begin{theorem}
		\label{relaxedMP} [The relaxed stochastic maximum principle]\label{thmrelaxed}%
		Let \ $\mu^{\ast}$ be an optimal relaxed control minimizing the functional $J$
		over $\mathcal{R}$, and let $x_{t}^{\mu^{\ast}}$be the corresponding optimal
		trajectory. Then there exists a unique triple of square integrable and adapted
		processes $(p^{\ \mu^{\ast}},q^{\ \mu^{\ast}},r^{\ \mu^{\ast}})$ which is the
		solution of the backward SDE
		\begin{equation}
			\left\{
			\begin{array}
				[c]{l}%
				\left.
				\begin{array}
					[c]{l}%
					dp_{t}^{\mu^{\ast}}=-\left\{
					{\displaystyle\int\limits_{A}}
					h_{x}(t,x_{t}^{\mu^{\ast}},a)\mu_{t}^{\ast}(da)+%
					{\displaystyle\int\limits_{A}}
					p_{t}^{n}b_{x}(t,x_{t}^{\mu^{\ast}},a)\mu_{t}^{\ast}(da)\right. \\
					\left.  +%
					{\displaystyle\int\limits_{A}}
					{\displaystyle\int\limits_{\Gamma}}
					r_{t}^{\mu^{\ast}}(\theta)f(t,x_{t^{-}}^{\mu^{\ast}},\theta,a)\mu_{t}^{\ast
					}\otimes\upsilon(da,d\theta)\right\}  dt\\
					\left.  -\left\{  \gamma_{x}(t,x_{t}^{\mu^{\ast}},a)p_{t}^{\mu^{\ast}}\mu
					_{t}^{\ast}(da)+q_{t}^{\mu^{\ast}}\sigma_{x}(t,x_{t}^{\mu^{\ast}})\right\}
					d\langle B\rangle_{t}+q_{t}^{\mu^{\ast}}dB_{t}\right. \\
					\left.  +%
					{\displaystyle\int\limits_{A}}
					{\displaystyle\int\limits_{\Gamma}}
					r_{t}^{\mu^{\ast}}(\theta)\widetilde{N}^{\mu^{\ast}}(dt,d\theta,da)+dk_{t}%
					^{\mu^{\ast}}\right.
				\end{array}
				\right. \\
				p_{T}^{\mu^{\ast}}=g_{x}(x_{T}^{\mu^{\ast}}),\text{ }k_{0}^{\mu^{\ast}}=0
			\end{array}
			\right.  \label{pmu}%
		\end{equation}
		such that if we assume that $b=0$ and $h=0$, then for all $\nu\in U$%
		\begin{align}
			0  &  \leq\widehat{\mathbb{E}}\left[  H(t,x_{t}^{\mu^{\ast}},\nu_{t}%
			,p_{t}^{\mu^{\ast}},q^{\mu^{\ast}},r_{t}^{\mu^{\ast}}(.))-%
			{\displaystyle\int\limits_{\Gamma}}
			H(t,x_{t}^{\mu^{\ast}},a,p_{t}^{\mu^{\ast}},q^{\mu^{\ast}},r_{t}^{\mu^{\ast}%
			}(.))\mu_{t}^{\ast}(da)\right. \label{main}\\
			&  \left.  +G^{\mu^{\ast}}(t,x_{t}^{\ast},u_{t}^{\ast},p,q,r(.))\right]
			\text{ \ \ \ \ \ }dt-a.e\nonumber
		\end{align}
		
	\end{theorem}
	
	The proof of this theorem is based on the following stability result of
	G-BSDEs with jumps. Note that this theorem is proved in the classical problems
	by Hu and Peng \cite{Hu}, and by H.Ben Gherbal and B.Mezerdi \cite{HBG} in the
	case with jump.
	
	\subsubsection{Stability theorem for G-BSDE's with jump}
	
	Let us denote by $M_{G}^{2}\left(  0,T\right)  $ the subset of $\mathcal{L}%
	_{G}^{2}\left(  0,T\right)  $ consisting of $\mathcal{F}_{t}-$progressively
	measurable processes. Consider the following G-BSDE with jump depending on a
	parameter $n$. Using the fact that under $\mathbb{P}\in\mathcal{P}$,
	$\widetilde{N}$ is a martingale and $B$ is a continuous martingale whose
	quadratic variation process $\langle B\rangle$ is such that $\pi_{t}%
	=\frac{d\langle B\rangle_{t}}{dt}$ is bounded by a deterministic $d\times d$
	symmetric positive definite matrix $\overline{\sigma},$and $p_{t}^{n}$ satisfy%
	\begin{equation}
		\left\{
		\begin{array}
			[c]{l}%
			\left.
			\begin{array}
				[c]{l}%
				dp_{t}^{n}=-\left\{  h_{x}(t,x_{t}^{n},u_{t}^{n})+p_{t}^{n}\left(
				b_{x}(t,x_{t}^{n},u_{t}^{n})-\pi_{t}\gamma_{x}(t,x_{t}^{\ast},u_{t}^{\ast
				})\right)  -\pi_{t}q_{t}^{n}\sigma_{x}(t,x_{t}^{n})\right. \\
				\left.  +%
				{\displaystyle\int\limits_{\Gamma}}
				r_{t}^{n}(\theta)f(t,x_{t^{-}}^{n},\theta,u_{t}^{n})\upsilon(d\theta)\right\}
				dt+\left.  q_{t}^{n}dB_{t}+%
				{\displaystyle\int\limits_{\Gamma}}
				r_{t}^{n}(\theta)\widetilde{N}(dt,d\theta)+dk_{t}^{n}\right.
			\end{array}
			\right. \\
			p_{T}^{n}=g_{x}(x_{T}^{n});\text{ }k_{0}^{n}=0.
		\end{array}
		\right.  \label{G-stability}%
	\end{equation}

	Then we have
	\[
	p_{t}^{n}=p_{T}^{n}+%
	{\displaystyle\int\limits_{t}^{T}}
	F^{n}(s,p_{s}^{n},q_{s}^{n},r_{s}^{n})ds-%
	{\displaystyle\int\limits_{t}^{T}}
	q_{s}^{n}dB_{s}-%
	{\displaystyle\int\limits_{t}^{T}}
	{\displaystyle\int\limits_{\Gamma}}
	r_{s}^{n}(\theta)N^{n}(ds,d\theta)-K_{T}^{n}+K_{t}^{n}\text{ \ \ \ }%
	t\in\left[  0;T\right]  .
	\]

	with
	\begin{align*}
		&  \left.  F^{n}(s,p_{s}^{n},q_{s}^{n},r_{s}^{n})=-h_{x}(t,x_{t}^{n},u_{t}%
		^{n})+p_{t}^{n}\left(  b_{x}(t,x_{t}^{n},u_{t}^{n})-\pi_{t}\gamma_{x}%
		(t,x_{t}^{\ast},u_{t}^{\ast})\right)  \right. \\
		&  \left.  -\pi_{t}q_{t}^{n}\sigma_{x}(t,x_{t}^{n})+%
		{\displaystyle\int\limits_{\Gamma}}
		r_{t}^{n}(\theta)f(t,x_{t^{-}}^{n},\theta,u_{t}^{n})\upsilon(d\theta)\right.
		.
	\end{align*}

	Using the linearity of the adjoint equation, it is not difficult to check that
	the following assumptions are verified :
	
	\begin{enumerate}
		\item For any $n$, $(p,q,r)\in%
		\mathbb{R}
		^{m}\times%
		\mathbb{R}
		^{m\times d}\times%
		\mathbb{R}
		,$ $F^{n}(.,p,q,r)\in M_{G}^{2}\left(  0,T\right)  $ and $p_{T}^{n}%
		\in\mathcal{L}_{G}^{2}\left(  0,T\right)  .$
		
		\item There exists a constant $C_{0}>0$ such that
		\begin{align*}
			&  \left\vert F^{n}(s,p_{1},q_{1},r_{1})-F^{n}(s,p_{2},q_{2},r_{2})\right\vert
			\\
			&  \leq C_{0}\left(  \left\vert p_{1}-p_{2}\right\vert +\left\vert q_{2}%
			-q_{2}\right\vert +%
			{\displaystyle\int\limits_{\Gamma}}
			\left\vert r_{1}-r_{2}\right\vert \upsilon(d\theta)\right)  \text{
				\ \ }P.a.s\text{ \ }a.e\text{ \ }t\in\left[  0;T\right]  ,
		\end{align*}

		\item $E\left(  \left\vert p_{T}^{n}-p_{T}^{\ast}\right\vert ^{2}\right)
		\overrightarrow{n\rightarrow\infty}$ $0,$
		
		\item $\forall t\in\left[  0;T\right]  ,$%
		\[
		\lim_{n\rightarrow\infty}\mathbb{E}^{\mathbb{P}}\left[  \left\vert
		{\displaystyle\int\limits_{t}^{T}}
		\left(  F^{n}(s,p_{s}^{\ast},q_{s}^{\ast},r_{s}^{\ast})-F^{\ast}(s,p_{s}%
		^{\ast},q_{s}^{\ast},r_{s}^{\ast})\right)  ds\right\vert ^{2}\right]  =0
		\]
		
	\end{enumerate}
	
	\begin{theorem}
		[Stability theorem for G-BSDE's with jumps]\label{stabilityBSDE}Let
		$(p^{\ n},q^{\ n},r^{\ n})$ and $(p^{\ \ast},q^{\ \ast},r^{\ \ast}),$ be the
		solutions of (\ref{pnn}) and (\ref{pmu}), respectively. We have
		\[
		\lim_{n\rightarrow\infty}\widehat{\mathbb{E}}\left[  \left\vert p^{n}-p^{\ast
		}\right\vert ^{2}+%
		{\displaystyle\int\limits_{t}^{T}}
		\left\vert q^{n}-q^{\ast}\right\vert ^{2}ds+%
		{\displaystyle\int\limits_{t}^{T}}
		{\displaystyle\int\limits_{\Gamma}}
		\left\vert r^{n}-r^{\ast}\right\vert ^{2}\upsilon(d\theta)ds+\left\vert
		k^{n}-k^{\ast}\right\vert ^{2}\right]  =0.
		\]
		
	\end{theorem}
	
	\begin{proof}
		Under every $\mathbb{P}\in\mathcal{P},$ we have
		\begin{align*}
			&  \mathbb{E}^{\mathbb{P}}\left\vert p_{t}^{n}-p_{t}^{\ast}\right\vert ^{2}+%
			{\displaystyle\int\limits_{t}^{T}}
			\left\vert q_{s}^{n}-q_{s}^{\ast}\right\vert ^{2}ds+%
			{\displaystyle\int\limits_{t}^{T}}
			{\displaystyle\int\limits_{\Gamma}}
			\left\vert r_{s}^{n}-r_{s}^{\ast}\right\vert ^{2}\upsilon(d\theta)ds\\
			&  \leq2\mathbb{E}^{\mathbb{P}}\left\vert \alpha_{t}^{n}\right\vert ^{2}\\
			&  +2\mathbb{E}^{\mathbb{P}}\left(
			{\displaystyle\int\limits_{t}^{T}}
			\left[  F^{n}(s,p_{s}^{n},q_{s}^{n},r_{s}^{n})-F^{n}(s,p_{s}^{\ast}%
			,q_{s}^{\ast},r_{s}^{\ast})\right]  ds\right)  ^{2}\\
			&  \leq2\mathbb{E}^{\mathbb{P}}\left\vert \alpha_{t}^{n}\right\vert
			^{2}+2(T-t)E%
			{\displaystyle\int\limits_{t}^{T}}
			\left\vert F^{n}(s,p_{s}^{n},q_{s}^{n},r_{s}^{n})-F^{n}(s,p_{s}^{\ast}%
			,q_{s}^{\ast},r_{s}^{\ast})\right\vert ^{2}ds
		\end{align*}

		with
		\[
		\alpha_{t}^{n}=p_{T}^{n}-p_{T}^{\ast}+%
		{\displaystyle\int\limits_{t}^{T}}
		\left[  F^{n}(s,p_{s}^{\ast},q_{s}^{\ast},r_{s}^{\ast})-F^{\ast}(s,p_{s}%
		^{\ast},q_{s}^{\ast},r_{s}^{\ast})\right]  ds+\left(  k_{T}^{\ast}-k_{T}%
		^{n}\right)  +\left(  k_{t}^{\ast}-k_{t}^{n}\right)  .
		\]

		Because of the assumption $2$, we get
		\begin{align}
			&  \left.  \mathbb{E}^{\mathbb{P}}\left\vert p_{t}^{n}-p_{t}^{\ast}\right\vert
			^{2}\leq\frac{2}{3}\mathbb{E}^{\mathbb{P}}\left\vert \alpha_{t}^{n}\right\vert
			^{2}+\frac{1}{6}%
			{\displaystyle\int\limits_{t}^{T}}
			\mathbb{E}^{\mathbb{P}}\left\vert p_{s}^{n}-p_{s}^{\ast}\right\vert
			^{2}ds\right. \label{pn}\\
			&  \left.  \mathbb{E}^{\mathbb{P}}%
			{\displaystyle\int\limits_{t}^{T}}
			\left\vert q_{s}^{n}-q_{s}^{\ast}\right\vert ^{2}ds\leq\frac{4}{3}%
			\mathbb{E}^{\mathbb{P}}\left\vert \alpha_{t}^{n}\right\vert ^{2}+\frac{2}{3}%
			{\displaystyle\int\limits_{t}^{T}}
			\mathbb{E}^{\mathbb{P}}\left\vert p_{s}^{n}-p_{s}^{\ast}\right\vert
			^{2}ds\right. \label{qn}\\
			&  \left.  \mathbb{E}^{\mathbb{P}}%
			{\displaystyle\int\limits_{t}^{T}}
			{\displaystyle\int\limits_{\Gamma}}
			\left\vert r_{s}^{n}-r_{s}^{\ast}\right\vert ^{2}\upsilon(d\theta)ds\leq
			\frac{4}{3}\mathbb{E}^{\mathbb{P}}\left\vert \alpha_{t}^{n}\right\vert
			^{2}+\frac{2}{3}%
			{\displaystyle\int\limits_{t}^{T}}
			\mathbb{E}^{\mathbb{P}}\left\vert p_{s}^{n}-p_{s}^{\ast}\right\vert
			^{2}ds\right.  . \label{rn}%
		\end{align}

		By the assumptions $3$, $4$ and the stability theorem of G-BSDE without jump,
		see \cite{HuLiHima}, we deduce that $\underset{n\rightarrow\infty
		}{lim}\mathbb{E}^{\mathbb{P}}\left\vert \alpha_{t}^{n}\right\vert ^{2}=0,$
		then $\underset{n\rightarrow\infty}{lim}\mathbb{E}^{\mathbb{P}}\left\vert
		p_{t}^{n}-p_{t}^{\ast}\right\vert ^{2}=0$ and $\underset{n\rightarrow
			\infty}{lim}\mathbb{E}^{\mathbb{P}}%
		{\displaystyle\int\limits_{t}^{T}}
		\left\vert q_{s}^{n}-q_{s}^{\ast}\right\vert ^{2}ds=0.$ Hence, by (\ref{rn})
		we get
		\[
		\underset{n\rightarrow\infty}{lim}\mathbb{E}^{\mathbb{P}}%
		{\displaystyle\int\limits_{t}^{T}}
		{\displaystyle\int\limits_{\Gamma}}
		\left\vert r^{n}-r^{\ast}\right\vert ^{2}\upsilon(d\theta)ds=0.
		\]
		Finally, by the aggregation property we conclude the desired result.
	\end{proof}
	
	\begin{proof}
		[Proof of Theorem \eqref{relaxedMP} ]By passing to the limit in inequality
		(\ref{main2}), and using lemma (\ref{Eke}), we get easily the inequality
		(\ref{main}).
	\end{proof}

\end{document}